\documentclass[10pt,a4paper]{amsart} 
\usepackage[draft]{optional}
\usepackage[british,american]{babel}
\usepackage[margin=0cm]{geometry}
\usepackage{lipsum}

\usepackage{mathrsfs}
\usepackage{amsfonts, amsmath, wasysym, caption}

\usepackage{amssymb,amsthm, paralist}

\usepackage{amsopn}
\usepackage{accents,bbm,bibgerm,dsfont,eucal,esint,paralist,url,verbatim,wasysym}
\usepackage[normalem]{ulem}
\usepackage{bbold}

\usepackage{
latexsym,
nicefrac
}

\usepackage[usenames]{color}
\usepackage{tikz}
\usetikzlibrary{decorations.pathreplacing}
\usepackage{url}

\definecolor{darkgreen}{rgb}{0,0.5,0}
\definecolor{darkred}{rgb}{0.7,0,0}

\usepackage[colorlinks, 
citecolor=darkgreen, linkcolor=darkred
]{hyperref}

\hypersetup{
 pdfauthor={Nadine Grosse, Melanie Rupflin},
 pdftitle={Holomorphic quadratic differentials dual to Fenchel-Nielsen coordinates},
 pdflang={en}
}

\usepackage{esint}
\usepackage{bibgerm}

\theoremstyle{plain}
\newtheorem{lemma}{Lemma}[section]
\newtheorem{thm}[lemma]{Theorem}
\newtheorem{prop}[lemma]{Proposition}
\newtheorem{cor}[lemma]{Corollary}
\theoremstyle{definition}

\newtheorem{rmk}[lemma]{Remark}

\numberwithin{equation}{section}

\newcommand{\ga}{\gamma}

\newcommand{\de}{\delta}

\newcommand{\Om}{\Omega}

\newcommand{\la}{\lambda}
\newcommand{\La}{\Lambda}

\newcommand{\si}{\sigma}

\newcommand{\Si}{\Sigma}
\renewcommand{\th}{\theta}
\newcommand{\Th}{\Theta}

\newcommand{\Ups}{\Upsilon}

\newcommand{\R}{\ensuremath{{\mathbb R}}}
\newcommand{\N}{\ensuremath{{\mathbb N}}}

\newcommand{\C}{\ensuremath{{\mathbb C}}}

\newcommand{\Hyp}{\ensuremath{{\mathbb H}}}

\DeclareMathOperator{\inj}{inj}

\newcommand{\norm}[1]{\Vert#1\Vert}  
\newcommand{\bnorm}[1]{\big\Vert#1\big\Vert}

\newcommand{\arsinh}{{\rm arsinh}}

\newcommand{\beq}{\begin{equation}}
\newcommand{\eeq}{\end{equation}}
\newcommand{\beqw}{\begin{equation*}}
\newcommand{\eeqw}{\end{equation*}}
\newcommand{\beqs}{\begin{equation*}}
\newcommand{\eeqs}{\end{equation*}}
\newcommand{\beqa}{\begin{equation}\begin{aligned}}
\newcommand{\eeqa}{\end{aligned}\end{equation}}
\newcommand{\beqas}{\begin{equation*}\begin{aligned}}
\newcommand{\eeqas}{\end{aligned}\end{equation*}}
\newcommand{\brmk}{\begin{rmk}}
\newcommand{\ermk}{\end{rmk}}
\newcommand{\partref}[1]{\hbox{(\csname @roman\endcsname{\ref{#1}})}}
\newcommand{\half}{\frac{1}{2}}

\newcommand{\dist}{\text{dist}}

\renewcommand{\i}{\mathrm{i}}
\newcommand{\pt}{\partial_t}
\newcommand{\M}{\ensuremath{{\mathcal M}}_{-1}}

\newcommand{\abs}[1]{\vert#1\vert} 
\newcommand{\babs}[1]{\left| #1\right|}

\newcommand{\eps}{\varepsilon}
\newcommand{\na}{\nabla}

\newcommand{\Hol}{{\mathcal{H}}} 

\newcommand{\Col}{\mathcal{C}}
\newcommand{\thin}{\text{-thin}}
\newcommand{\thick}{\text{-thick}}

\newcommand{\Rea}{\mathrm{Re}\,}

\newcommand{\ov}{\overline}
\newcommand{\supp}{\text{supp}}
\newcommand{\ddt}{\frac{d}{dt}}

\newcommand{\thalf}{\tfrac12}

\newcommand{\Area}{\text{Area}}
\newcommand{\define}{\mathrel{\mathrm {:=}}}

\title
{{\sc
Holomorphic quadratic differentials dual to Fenchel-Nielsen coordinates
}
\\ 
}
\author{Nadine Gro{\ss}e and Melanie Rupflin}

\begin{document}

\begin{abstract}
We discuss bases of the space of holomorphic quadratic differentials that are dual to the differentials of Fenchel-Nielsen coordinates and hence appear naturally when considering functions on the set of
hyperbolic metrics which are invariant under pull-back by diffeomorphisms, such
as eigenvalues of the Laplacian. The precise estimates derived in the current paper form the basis for the proof of the sharp eigenvalue estimates on degenerating surfaces obtained in \cite{eigenvalue}.
\end{abstract}
\maketitle
\vspace{-1cm}
\section{Introduction and results}\label{sect:intro}
Let $M$ be a closed oriented surface of genus $\gamma\geq 2$ (always assumed to be connected) and let $\M$ be the set of smooth hyperbolic metrics on $M$. We recall that the tangent space to $T_g\M$ splits orthogonally
$$T_g\M= \{L_Xg, X\in \Gamma(TM)\}\oplus H(g)$$
into the directions generated by the pull-back by diffeomorphisms and the \textit{horizontal space}
$H(g)=\Rea(\Hol(M,g))$, which is given by the real part of the  complex vector space 
$$\Hol(M,g)\define \{\Psi\ |\ \text{ holomorphic quadratic differentials on } (M,g)\}$$
whose (complex) dimension is $3(\gamma-1)$. 

In the present paper we analyse bases of $H(g)$ which appear naturally if one studies functions $f\colon\M\to\R$ that are defined in terms of geometric properties of $(M,g)$, such as eigenvalues of differential operators on $(M,g)$, as such functions are of course invariant under the pull-back by diffeomorphisms. It is natural to view any such $f$ as a function of the Fenchel-Nielsen coordinates $\{\ell_j,\psi_j\}$ of $(M,g)$ whose definition we recall below. The derivatives of such a function with respect to the Fenchel-Nielsen coordinates are then determined in terms of the $L^2$-gradient  of $f$, which is itself an element of $H(g)$ as the above splitting is orthogonal, 
 and the dual bases $\{\La^j,\Psi^j\}$ to the differentials of the Fenchel-Nielsen coordinates that we consider in the present paper. 

To formulate our results,  we first recall that any hyperbolic surface $(M,g)$ can be decomposed
into pairs of pants by cutting along a family  $\{\si^j\}_{j=1}^{3(\gamma-1)}$ of pairwise disjoint simple closed geodesics and that the metric on a pair of pants is uniquely determined (up to pull-back by diffeomorphisms) by the lengths of its three boundary curves. Keeping track of how the decomposing geodesics were chosen, the hyperbolic metric $g$ on $M$ is hence determined (up to pull-back by diffeomorphisms) by the Fenchel-Nielsen coordinates $\{\ell_j,\psi_j\}_{j=1}^{3(\gamma-1)}$, consisting of the 
 length coordinates $\ell_j=L_g(\si^j)$ of the geodesics along which we cut,
and the  twist coordinates $\psi_j$  which describe how the pairs of pants are glued together along $\si^j$.

We then note that for a function $f\colon\M\to \R$ as considered above 
the derivatives of $f$ with respect to the Fenchel-Nielsen coordinates are given by 
\beq \label{eq:derivative} \frac{\partial f}{\partial \ell_j}=\langle \nabla f, \Rea \Lambda^j\rangle\ \text{ and }  \frac{\partial f}{\partial \psi_j}=\langle \nabla f, \Rea \Psi^j\rangle\eeq
where 
 $\Psi^j, \La^j$ of $\Hol(M,g)$ are the elements of $\Hol(M,g)$ that are 
 dual to the (real differentials of the) Fenchel-Nielsen coordinates $\psi_j, \ell_j$ in the sense that for every $i,j$
\beq \label{def:Psi_La}
d\ell_j(\Rea(\La^i))=\de_j^i=d\psi_j(\Rea(\Psi^i)) \text{ and } 
d\psi_j(\Rea(\La^i))=0=d\ell_j(\Rea(\Psi^i)).
\eeq 
Here and in the following $d\ell_j\colon T_g\M\to \R$ is
given by $d\ell_j(k)=\frac{d}{d\eps}|_{\eps=0} L_{g_\eps}(\si^j(\eps))$ where  $g_\eps$ is a smooth curve of metrics in $\M$ so that $\partial_\eps\vert_{\eps=0}g_\eps=k$
 and  $\si^j(\eps)$ is the unique simple closed geodesic in $(M,g_\eps)$ homotopic to $\si^j$, with $d\psi_j$ defined in the same way. 

The main results of this paper, i.e. Theorem \ref{thm:La} and \ref{thm:Dehn-twists-main}, yield precise estimates on this dual basis of $H(g)$. In situation where the $L^2$-gradient of $f$ is known, e.g. in the case of simple eigenvalues of the Laplacian  compare \cite[Lemma 2.4]{eigenvalue}, such estimates can be combined with properties of $\na f$ to study the dependences of $f$ on the geometry of $(M,g)$. In particular, the results we prove in the current paper  play a crucial role in our proof of sharp estimates on the principal eigenvalue on degenerating hyperbolic surfaces in \cite{eigenvalue}.

Before we turn to our analysis of $\{\La^j,\Psi^j\}$, it is useful to first consider a related basis of $\Hol(M,g)$ 
which is dual to the \textit{complex} differentials of the 
length coordinates 
$$\partial \ell_j\colon\Hol(M,g)\to \C,$$
which were introduced in \cite[Remark 4.1]{RT-neg} and \cite{Wolpert82} as follows:
We view $\Hol(M,g)$ as real vector space with complex structure $J$ and identify $\Hol(M,g)$ with a subspace of $T_g\M$ via the isomorphism $\Phi\mapsto \Rea(\Phi)$ and define
\beq 
\partial \ell_j(\Phi)\define \thalf \big(d\ell_j(\Phi)-\i d\ell_j(J\Phi)\big)
\label{def:dell} \eeq

As \cite[Theorem~3.7]{Wolpert82} assures that $\Upsilon\mapsto (\partial\ell_1,\ldots, \partial \ell_{3(\gamma-1)})$ is an isomorphism from $\Hol(M,g)$ to $\C^{3(\gamma-1)}$,  we may consider the corresponding dual basis $\{\Th^j\}$ of $\Hol(M,g)$, and the corresponding renormalised basis $\{\Om^j\}$, which are 
characterised by 
\beq 
\label{def:Om-by-Th}
\partial\ell_j(\Th^i)=\de_j^i, \text{ respectively } 
\Om^j\define -\frac{\Th^j}{\norm{\Th^j}_{L^2(M,g)}}.
\eeq

 \begin{prop} \label{prop:RT-new} Let $(M,g)$ be any closed oriented hyperbolic surface of genus $\gamma\geq 2$ and let $\mathcal{E}=\{\si^1, \ldots, \si^{3(\gamma-1)}\}$ be any set of simple closed geodesics that decompose $(M,g)$ into pairs of pants and let  
 $\eta\in (0,\arsinh(1))$ and $\bar L<\infty$ be so that 
\begin{align}
 \mathcal{E}\ \textrm{contains all simple closed geodesics\ } \si \textrm{ of } (M,g) 
 \textrm{ of length } L_g(\si)\leq 2\eta\label{ass:eta}
\end{align}
and
\begin{align}\label{ass:upperbound}
\ L_g(\si)\leq \bar L \textrm{ for every } \si\in \mathcal{E}.
\end{align} 
 Then, there exist constants 
 $C$ and $\eps_1>0$ that depend only on the genus, $\eta$ and $\bar L$ so that for every $j=1,\ldots, 3(\gamma-1)$, 
 the  elements 
 $\{\Om^j\}_{j=1}^{3(\gamma-1)}$ defined in \eqref{def:Om-by-Th} above, satisfy
\begin{align}
\label{est:RT-neg1-new}
 \Vert  {\Om}^j\Vert_{L^\infty(M\setminus \Col( {\si}^j), g)} + 
 \Vert  {\Om}^j - b_0( {\Om}^j, \Col( {\si}^j))dz^2\Vert_{L^\infty(\Col (\sigma^j), g)} &\leq C \ell_j^{3/2},\\
  \label{est:RT-neg3-new}
1-C\ell_j^3\leq b_0( {\Om}^j, \Col( {\si}^j)) \Vert dz^2\Vert_{L^2(\Col( {\si}^j),g)} &\leq \, 1,\\
  \label{est:RT-neg4-new}
 |\langle  {\Om}^i,  {\Om}^j\rangle_{L^2(M,g)}|&\leq \, C\ell_i^{3/2} \ell_j^{3/2} \text{ for every } i\neq j
\end{align} 
as well as 
\beq
\label{est:RT-neg3-newest}
b_0( {\Om}^j, \Col( {\si}^j)) \Vert dz^2\Vert_{L^2(\Col( {\si}^j),g)}\geq \eps_1>0.
\eeq

Here $\Col( {\si}^j)\subset M$ denotes the collar around the geodesic $\si^j$, cf. Lemma \ref{lemma:collar}, and $b_0(\cdot,\Col(\si^j))dz^2$ denotes the principal part of the Fourier expansion \eqref{eq:Laurent} on the collar $\Col(\si^j)$.
\end{prop}

We remark that bases of the space of holomorphic quadratic differentials which are related to Fenchel-Nielsen and other choices of coordinates on Teichm\"uller space have been considered by many authors and we refer in particular to the
works of 
Masur \cite{Masur}, Yamada \cite{Yamada1,Yamada2}, Wolpert ~\cite{Wolpert82, WII, Wolpert03, Wolpert08, Wolpert12} and the recent work of Mazzeo-Swoboda \cite{Mazzeo-Swoboda} and the references therein for an overview of existing results.
For the present work a comparison with the paper \cite{Wolpert82, WII, Wolpert08, Wolpert12}  of Wolpert, the work \cite{Masur} of Masur and the previous joint work \cite{RT-neg} of Topping and the second author is particularly relevant. The above Proposition \ref{prop:RT-new} can be seen as an analogue of Wolpert's Theorem 8 and Corollary 9 of \cite{Wolpert12}, see also \cite[Lemma~3.12]{Wolpert08}, on the behaviour of a gradient basis, but in the setting of dual bases. To be more precise,
in \cite[Corollary~9]{Wolpert12}, see also \cite[Lemma~3.12]{Wolpert08}, Wolpert proves error estimates for the elements $\text{grad }\ell_j$ which represent the gradient of the length functionals $d\ell_j$, i.e.~are characterised by $d \ell_j(k)=\langle \text{grad }\ell_j,k\rangle$ for any $k\in H(g)=\Rea(\Hol(M,g))$, which have the same optimal order of errors as our estimates (amounting to error rates of $O(\ell_j^{3/2})$ when elements are normalised to have unit $L^2$-norm).  The motivation of the current paper is not to improve or reprove results on such gradient bases, which have been very influential in the study of Teichm\"uller space, but rather to develop the relevant results for dual bases, first for the complex differentials $\partial \ell_j$ and then, more importantly, for the real differentals $(d\ell_j,d\psi_j)$.

We stress that it is control on dual bases, rather than gradient bases, that is essential for the applications that motivate the present paper, namely studying the dependence of functions, such as eigenvalues, on the Fenchel-Nielsen coordinates, compare \eqref{eq:derivative} and see \cite{eigenvalue} for a first instance of such an application. 

We also remark that while we could have tried to derive bounds on the dual basis from bounds on the corresponding gradient basis as obtained in \cite{Wolpert08, Wolpert12}, any such proof would require quantative control on the inverse of the matrix formed by all the inner products $\langle \text{grad } \ell_j, \text{grad } \ell_k\rangle$ of the elements of the gradient basis and, as we shall explain further in Remark \ref{rmk:inverse}, obtaining the necessary control on this inverse is essentially equivalent to the main step of our proof of Proposition \ref{prop:RT-new}. We instead base our proof of Proposition \ref{prop:RT-new} on the results obtained in the joint work \cite{RT-neg} of Topping and the second author that we recall in  Lemma \ref{lemma:RT-neg} below, 
where the above estimates were proven for elements of the corresponding  dual basis $\tilde \Om^1,\ldots,\tilde\Om^k$ of $(\ker(\partial\ell_1, \ldots, \partial\ell_k))^\perp$, for $\si^1,\ldots,\si^j$ 
the geodesics of $(M,g)$ whose length is no more than a sufficiently small constant. 
We also note that while some dual bases of coordinates had already been considered by Masur, the results of \cite{Masur} do not yield the quantitative control on the dual basis that is needed in applications such as \cite{eigenvalue}.

We note that  Proposition \ref{prop:RT-new} implies in particular that 
 for every $\de>0$ there exists $C_\de=C(\de, \eta, \bar L, \gamma)$ so that 
\beq 
\label{est:Linfty_Om_thick}
\norm{\Om^j}_{L^\infty(\de\thick(M,g))}\leq C_\de\ell_j^{3/2} \text{ while }
\norm{\Om^j}_{L^\infty(M,g)}\leq C\ell_j^{-1/2},
\eeq
see \cite{RT-neg} for the analogue results on the $\tilde\Om^j$. Here we note that the second inequality in \eqref{est:Linfty_Om_thick} is obtained by combining the fact that $\norm{dz^2}_{L^{\infty}(\Col(\si^j))}\leq C\ell_j^{-2}\leq C\ell_j^{-1/2} \norm{dz^2}_{L^2(\Col(\si^j))}$, compare \eqref{est:sizes_on_collars} and \eqref{est:dz-lower}, with
 \eqref{est:RT-neg1-new} and  \eqref{est:RT-neg3-new}.
 
We also remark that the lower bound \eqref{est:RT-neg3-newest} is equivalent to an upper bound for the  elements $\Th^j$ of the original dual basis of $\norm{\Th^j}_{L^2(M,g)}\leq C\ell_j^{-\half}$, compare Lemma \ref{lemma:Om}, and hence that 
\beq 
 \label{est:Linfty_Om_global}
\norm{\Th^j}_{L^\infty(\de\thick(M,g))}\leq C_\de\ell_j. 
\eeq

We now turn back to the analysis of the elements $\{\La^j,\Psi^j\}$ which are dual to the real differentials of Fenchel-Nielsen coordinates and hence appear when considering the dependence of functions on Fenchel-Nielsen coordinates. We first remark that while by definition 
$d\ell_j(\tfrac12\Rea\Th^i)=\de_{j}^i=d\ell_j(\Rea\La^i) \text{ for every } 1\leq i,j\leq 3(\gamma-1)$, compare also \eqref{eq:length-evol} and \eqref{eq:dell-c}, the elements $\half\Rea\Th^i$ will in general not leave the twist coordinates invariant and hence not agree with $\Rea\La^i$. However, we shall see that the difference between these elements is only of order $O(\ell_j)=\norm{\La^j}_{L^2}\cdot O(\ell_j^{3/2})$, so that we also obtain error estimates such as the analogue of \eqref{est:RT-neg1-new} with this sharp error rate for the elements $\La^j$. To be more precise, we will show
\begin{thm}\label{thm:La}
 Let $(M,g)$ be any closed oriented hyperbolic surface of genus $\gamma$ and let $\mathcal{E}=\{\si^1, \ldots, \si^{3(\gamma-1)}\}$ be any decomposing set of simple closed geodesics. 
 Then there exists a constant $C$ that depends only on the genus and the numbers $\eta\in (0,\arsinh(1))$ and $\bar L<\infty$ for which \eqref{ass:eta} and \eqref{ass:upperbound} are satisfied so that 
the elements $\La^j$ which induce only a change of the length coordinate $\ell_j$ as specified in \eqref{def:Psi_La} are given by
\beqs
\label{eq:writing_La} 
\La^j=\tfrac12\Th^j+\sum_k \i \cdot c_k^j \Om^k
\eeqs
for coefficients $c_k^j\in \R$ which satisfy
\beq 
\label{est:coeff-lemma-La}\abs{c_k^j}\leq C\ell_j\ell_k^{3/2} \text{ for every } j,k=1,\ldots, 3(\gamma-1) .
\eeq
In particular 
\beq \label{est:La-minus-Th}
\norm{\La^j-\tfrac12\Th^j}_{L^\infty(M,g)}\leq C\ell_j.
\eeq
\end{thm}

For the elements $\Psi^j$ that generate only a Dehn-twist we recall that Wolpert's length-twist duality \cite[Theorem 2.10]{Wolpert82} establishes that 
$\Psi^j$ can be written in terms of the gradient of the corresponding length coordinate, which would give one route to obtain error estimates on these elements, by using the bounds on the $\text{grad }\ell_j$ proven in \cite{Wolpert08}. 
In applications, in particular to the study of eigenvalues as carried out in \cite{eigenvalue}, it is however very useful to be able to characterise the $\Psi^j$ in terms of 
the dual basis $\Om^j$ respectively $\Th^j$, as the principal parts of the $\Th^j$ are determined explicitly, namely $b_0(\Th^j,\Col(\si^i))=-\de^{ij} \frac{\ell_j}{\pi^2}$, compare \eqref{eq:dell-c}, and as
 this dual basis appears explicitly in the characterisation of the $L^2$-gradient of small eigenvalues on degenerating surfaces, with Theorem 2.5 of \cite{eigenvalue} e.g. establishing that 
$$\na \la\sim \frac{1}{8\pi \la} \Rea(\Th^1)$$
on surfaces with one degenerating disconnecting closed geodesic, compare also  
 Lemma \ref{lemma:Om}. We will hence furthermore prove  

\begin{thm}\label{thm:Dehn-twists-main}
 Let $(M,g)$ be any closed oriented hyperbolic surface of genus $\gamma$ and let $\mathcal{E}=\{\si^1, \ldots, \si^{3(\gamma-1)}\}$ be any decomposing set of simple closed geodesics. 
 Then there exists a constant $C$ that depends only on the genus and the numbers $\eta\in (0,\arsinh(1))$ and $\bar L<\infty$ for which \eqref{ass:eta} and \eqref{ass:upperbound} are satisfied so that the elements $\Psi^j$ which generate Dehn-twists as described in \eqref{def:Psi_La} are given by 
\beq
\label{eq:write-Psi-with-Om}
\frac{\Psi^j}{\Vert \Psi^j\Vert_{L^2(M,g)}}= -a_j \i\Om^j+\i\sum_{k\neq j} c_k^j \Om^k \text{ for some } a_j\in \R^+, c_{k}^j\in \R
\eeq
for coefficients
\beq
\label{est:coeff-Psi}
\abs{c_k^j}\leq C\ell_k^{3/2} \ell_j^{3/2} \text{ and } \abs{1-a_j}\leq C\ell_j^3,\eeq
in particular
  \begin{equation}\label{est:lemma-Psi1}
\Vert \tfrac{\Psi^j}{\Vert \Psi^j\Vert_{L^2(M,g)}} +\i \Om^j\Vert_{L^\infty (M,g)} \leq C\ell_j^{3/2},
 \end{equation}
where furthermore
 \begin{align}\label{est:lemma-Psi4}
  \left|  \Vert \Psi^j\Vert_{L^2(M,g)} - 8\pi \Vert dz^2\Vert_{L^2(\Col(\si^j),g)}^{-1}\right| \leq C\ell_j^{9/2}.
  \end{align}
\end{thm}

\section{Proofs of the results}
In this section  we prove our main results: In Section \ref{subsec:holo_main_1} we prove the properties of the dual basis $\{\Th^j\}$  of $\{\partial \ell_j\}$ stated in Proposition~\ref{prop:RT-new}.
Section~\ref{subsec:holo_twist} is then concerned with the analysis of 
the elements $\Psi^j$ which generate Dehn-twists and hence the proof of Theorem \ref{thm:Dehn-twists-main}, while the properties of the elements $\La^j$ which induce a change of only the length coordinates are analysed in Section~\ref{subsec:holo_la}, where we prove Theorem \ref{thm:La}. Before that we recall well-known properties and results of holomorphic quadratic differentials that are used throughout the proofs of our main results.

\subsection{Preliminaries: Properties of holomorphic quadratic differentials}\label{subsec:holo_main_1} $ $

Before we begin with the proofs of our main results, we recall some standard properties of holomorphic quadratic differentials as well as results on $\Hol(M,g)$ from the joint works \cite{RT2,RT-neg} and \cite{RTZ} of Topping (respectively Topping, Zhu)  
and the second author that we will use later on. 
We note that alternatively we could also use other bases of $\Hol(M,g)$, such as the gradient basis of the length coordinates considered by Wolpert in \cite{WII,Wolpert08}, as basis of our work.

We recall that a quadratic differential is a complex tensor $\Psi$ which is given in local isothermal coordinates $(x,y)$
as $\Psi=\psi \cdot dz^2$, $z=x+\i y$. Here $\psi$ is a complex function which for elements of $\Hol(M,g)$ is furthermore asked to be holomorphic. 
 Using the normalisation that $\abs{dz^2}_{g}=2\rho^{-2}$ for $g=\rho^2(dx^2+dy^2)$ we may write 
the  (hermitian) $L^2$-inner product on the space of quadratic differentials locally as  
\beq 
\label{eq:inner-prod-hallo} 
\langle\Psi,\Phi\rangle_{L^2}= \int \psi\cdot \bar \phi \abs{dz^2}_g^2dv_g= 4\int \psi\cdot \bar\phi \rho^{-2}dxdy.\eeq
In particular
\beq 
\label{eq:Re_inner_prod} 
\langle \Rea(\Psi), \Rea(\Phi)\rangle_{L^2(M,g)} =\thalf \Rea\langle \Psi, \Phi\rangle_{L^2(M,g)}, \eeq
where here and in the following we use the standard abuse of notation that all $L^2$ inner-products, be it of quadratic differentials as in \eqref{eq:inner-prod-hallo} or of real $(0,2)$ tensors  as in \eqref{eq:Re_inner_prod}, are denoted by the same notation $\langle\cdot, \cdot\rangle_{L^2}$.

We also recall that this relation implies that the projection $P_g^H$ from the space of symmetric real $(0,2)$-tensors
onto $H(g)=\Rea(\Hol(M,g))$ and the projection $P^\Hol_g$ from the space of $L^2$-quadratic differentials onto $\Hol(M,g)$ are related by 
$$P_g^H(\Rea(\Phi))= \Rea(P_g^\Hol(\Phi)).$$
We furthermore recall from  \cite[Proposition 4.10]{RT-neg}
that for any quadratic differential $\Upsilon$
\beq 
\label{est:L1-proj}
\norm{P^\Hol_g(\Upsilon)}_{L^1(M,g)}\leq C \norm{\Upsilon}_{L^1(M,g)}
\eeq
for a constant $C$ that depends only on the genus.

Let now $\Col(\si)$ be a collar around a simple closed geodesic $\si$ in $(M,g)$ described by the Collar lemma~\ref{lemma:collar} of Keen-Randol that we recall in the appendix. We we will often use that on $\Col(\si)$ we
may represent any $\Upsilon\in \Hol(M,g)$ by 
its Fourier series in collar coordinates $(s,\th)$ 
\beq\label{eq:Laurent}
\Upsilon= \sum_{n=-\infty}^\infty b_n(\Upsilon) e^{n(s+\i\theta)} dz^2,\qquad\quad b_n(\Upsilon)=b_n(\Upsilon, \Col(\si)) \in \C, \quad z=s+ \i \theta \eeq
and that on $\Col(\si)$ we may split $\Upsilon$ orthogonally into its principal part $b_0(\Upsilon)dz^2$ and its collar decay part $\Upsilon-b_0(\Upsilon)dz^2$. Hence, for any $\Upsilon, \Psi\in \Hol(M,g)$ 
\beqa \label{eq:orth-Four1}
\langle \Upsilon, \Psi\rangle_{L^2(\Col(\si))} 
&=b_0(\Upsilon)\cdot \ov{b_0(\Psi)} \norm{dz^2}_{L^2(\Col(\si))}^2 +\langle \Upsilon, \Psi-b_0(\Psi)dz^2\rangle_{L^2(\Col(\si))},
\eeqa
where here and in the following we sometimes abbreviate 
$b_0(\Psi)=b_0(\Psi,\Col(\si))$ respectively $b_0^i(\Upsilon)=b_0(\Upsilon,\Col(\si^i))$ if it is clear from the context that we work on a fixed collar respectively on collars around a fixed collection $\{\si^i\}$ of simple closed geodesics. We will also use the convention that norms over $\Col(\si)$ are always computed with respect to the hyperbolic metric $g=\rho^2(ds^2+d\th^2)$.

We recall that for every $\de>0$ we may bound an arbitrary
 element $\Upsilon\in \Hol(M,g)$ by 
\beq 
\label{est:Linfty-by-L1} 
\norm{\Upsilon}_{L^\infty(\de\thick(M,g))}\leq C_\de \norm{\Upsilon}_{L^2(M,g)},
\eeq
where $C_\de$ depends on $\de$ and the genus. Indeed, 
\cite[Lemma~2.6]{RT-horizontal} ensures that \eqref{est:Linfty-by-L1} holds true for $C_\de=C\de^{-1/2}$, $C$ depending only on the genus, and indeed also with the $L^\infty$-norm on the left hand side replaced by the $C^k$-norm (then with $C$ depending additionally on $k$).

We also recall that the collar regions around disjoint geodesics are disjoint,  that 
the $\arsinh(1)$ thin part of a hyperbolic surface is always contained in the union of the collars around the simple closed geodesics of length less than $2\arsinh(1)$, that such geodesics are always disjoint and that their number is no more than $3(\gamma-1)$. 

If $\{\si^1, \ldots, \si^k\}$ is the set of all simple closed geodesics of $(M,g)$ of length no more than some constant 
$2\eta<2\arsinh(1)$
we hence have that, as observed in \cite[Lemma~2.4]{RTZ},
\beq 
\label{est:W2}
\norm{w}_{L^\infty(M,g)}\leq C_\eta  \norm{w}_{L^1(M,g)}
\eeq
for all elements 
$w\in W_\eta\define\{\Upsilon\in \Hol(M,g): b_0(\Upsilon, \Col(\si^j))=0,\ 1\leq j\leq k\}.$ Here and in the following all constants are allowed to depend on the genus in addition to the indicated dependences unless explicitly said otherwise.

We also recall the well-known fact that along a curve $(g(t))_t$ of hyperbolic metrics with $g(0)=g$ and $\pt g(0)=\Rea \Upsilon$ for $\Upsilon \in\Hol(M,g)$ the evolution of the length $\ell(t)$ of the simple closed geodesic $\si_t\subset (M,g(t))$ homotopic to $\si_0$ is given by 
\beq 
\label{eq:length-evol}
\tfrac{d}{dt}\, \ell=-\tfrac{2\pi^2}{\ell} \Rea(b_0(\Upsilon, \Col(\si_0))) \qquad \text{ at } t=0,
\eeq
see e.g. \cite[Remark 4.12]{RT-neg} or \cite{Wolpert82}.
So, as observed in \cite[Remark 4.1]{RT-neg}, if we select any  $k$ disjoint simple closed geodesics $\si^j$ in $(M,g)$ 
we have
\beq 
\label{eq:ker-dell}
\ker(\partial \ell_1,\ldots,\partial \ell_k)=\{\Upsilon\in \Hol(M,g): b_0^j(\Upsilon)=b_0(\Upsilon, \Col(\si^j))=0 \text{ for } j=1,\ldots,k\},\eeq
where $\partial \ell_j$ is defined as in \eqref{def:dell} and thus given by
\beq
\label{eq:dell-c}
\partial \ell_j(\Upsilon)=\tfrac12 (-\tfrac{2\pi^2}{\ell_j} \Rea(b_0^j(\Upsilon))+\i \tfrac{2\pi^2}{\ell_j}\Rea(b_0^j(\i \Upsilon))) =-\tfrac{\pi^2}{\ell_j}b_0^j(\Upsilon).
\eeq
In particular for $\{\si^1, \ldots, \si^k\}$  chosen as above as the set of geodesics of length $\leq 2\eta$ we have that $W_\eta=\ker(\partial \ell_1,\ldots,\partial \ell_k)$ and recall that as a consequence of \cite[Theorem~3.7]{Wolpert82}, 
$$\text{codim}(W_\eta)=\text{codim}(\ker(\partial \ell_1,\ldots,\partial \ell_k))=k,$$
see also \cite{RTZ} for an alternative proof in case that $\eta$ is sufficiently small.
 
The fine properties of the elements of  $W_\eta^{\perp}$, $\eta$ small, were analysed in \cite{RT-neg} and we shall use in particular the following version of \cite[Lemma~4.5]{RT-neg}, compare \cite[Lemma~3.12]{Wolpert08} for a closely related result on the corresponding gradient basis.

\begin{lemma} \label{lemma:RT-neg} [Contents of \cite[Lemma~4.5]{RT-neg}]
For any genus $\gamma \geq 2$ there exists a number 
 $\eta_1 \in (0, \arsinh (1))$ so that for every $\bar \eta\in (0,\eta_1]$ the following holds true for a constant $C$ that depends only on $\bar \eta$ and the genus: 
\newline
Let $(M,g)$ be a closed oriented hyperbolic surface of genus $\gamma$ and let $\{ \si^1, \ldots, \si^k\}$ be the set of all simple closed geodesics in $(M,g)$ of length no more than $2 \bar \eta$. Define
\[ W=W_{\bar \eta}\define \{ \Upsilon\in \Hol(M,g)\ |\ \partial\ell_j(\Upsilon)=0, \quad j=1,\ldots, k\},\]
$\partial \ell_j$ the differentials of the length coordinates associated to $\si^j$, compare \eqref{def:dell}.
\newline
Then there exists a (unique) basis $\tilde{\Om}^1,\ldots, \tilde{\Om}^k$ of $W^\perp$, normalised by $\Vert \tilde{\Om}^j\Vert_{L^2(M,g)}=1$, so that 
\[ b_0(\tilde{\Om}^j, \Col(\si^i))=0\quad \text{for}\ i\neq j \text{ while }  b_0(\tilde{\Om}^j, \Col(\si^j))\in \R^+\]
and each $\tilde{\Om}^j$ is concentrated essentially only on the corresponding collar in the sense that
\begin{align}\label{est:RT-neg1}
  \Vert \tilde{\Om}^j\Vert_{L^\infty(M\setminus \Col(\si^j), g)} +
\Vert \tilde{\Om}^j - b_0(\tilde{\Om}^j, \Col(\si^j)) dz^2\Vert_{L^\infty(\Col(\si^j), g)}\leq C \ell_j^{3/2}
\end{align}
and
\begin{align}\label{est:RT-neg3}
 1-C\ell_j^3\leq b_0(\tilde{\Om}^j, \Col(\si^j)) \Vert dz^2\Vert_{L^2(\Col(\si^j),g)} \leq 1.
\end{align}
Furthermore, the $\tilde{\Om}^j$'s are nearly orthogonal in the sense that for $i\neq j$
\begin{align}\label{est:RT-neg4}
 |\langle \tilde{\Om}^i, \tilde{\Om}^j\rangle_{L^2(M,g)}|\leq C\ell_i^{3/2} \ell_j^{3/2}
\end{align}
and satisfy 
\beq \label{est:RT-neg-Linfty}
\norm{\tilde\Om^j}_{L^\infty(M,g)}\leq C\ell_j^{-1/2} \text{ with }
\norm{\tilde\Om^j}_{L^\infty(\de\thick(M,g))}\leq C_\de \ell_j^{3/2}
\eeq
for any $\de>0$, where $C_\de$ depends on $\de$, $\bar\eta$ and the genus.
\end{lemma}
We can view the $\tilde \Om^j$ as renormalisations  
$\tilde \Om^j=
-\frac{\tilde \Th^j}{\norm{\tilde \Th^j}_{L^2(M,g)}}$
of the dual basis $\{\tilde \Th^j\}$ of $W^\perp $ 
to $\{\partial \ell_j\}$, i.e. of the elements 
\beq\label{def:tilde-th}
\tilde \Th^j\in W^\perp \text{ for which } 
\de_{i}^j =  \partial \ell_i (\tilde \Th^j)=
-\tfrac{\pi^2}{\ell_j} b_0(\tilde\Th^j, \Col(\si^i)).\eeq

\begin{rmk}\label{rmk:half}
After possibly reducing $\eta_1=\eta_1(\gamma)$, we obtain that in the setting of 
 Lemma~\ref{lemma:RT-neg} 
\beq 
\label{est:b0-lower-half}
b_0(\tilde{\Om}^j, \Col(\si^j)) \Vert dz^2\Vert_{L^2(\Col(\si^j),g)}\geq \tfrac12 \text{ for every } j=1,\ldots, k\eeq 
and in the following we shall always use Lemma~\ref{lemma:RT-neg} for 
$\eta_1=\eta_1(\gamma)$ chosen in this way.
\end{rmk}

\begin{proof}[Proof of Remark \ref{rmk:half}]
Let $\eta_0=\eta_0(\gamma)$ be a number for which Lemma~\ref{lemma:RT-neg} holds true.
Given a hyperbolic surface $(M,g)$ we let $\{\si^1,\ldots,\si^k\}$ be the set of geodesics of length no more than $2\eta_0$, without loss of generality assumed to be ordered by increasing length, and denote by $\bar \Om^1,\ldots,\bar\Om^k$ the basis of $W_{\eta_0}^\perp$ obtained in that lemma. Let now $\bar\eta\in(0,\eta_1]$ for a number $\eta_1\leq \eta_0$ that is to be determined and let $k_1\leq k$ be so that the set of geodesics of length $2\bar \eta$ or less is $\{\si^1,\ldots, \si^{k_1}\}$ and let $\tilde \Om^1,\ldots, \tilde \Om^{k_1}$ be the basis of $W_{\bar \eta}$ from Lemma~\ref{lemma:RT-neg}. 
We note that while the $\tilde \Om^j $ satisfy all of the estimates stated in Lemma~\ref{lemma:RT-neg} the constants $C$ in these estimates depend on $\bar \eta$ so that we cannot directly conclude that \eqref{est:RT-neg3} implies \eqref{est:b0-lower-half} for sufficiently small $\eta_1$.
Instead we apply \eqref{est:RT-neg3} to the corresponding elements $\bar\Om^1,\ldots,\bar\Om^{k_1}$ of the basis of $W_{\eta_0}$, as this allows us to conclude that 
\beq 
\label{snow1}
b_0^j(\bar\Om^j)\norm{dz^2}_{L^2(\Col(\si^j))}\geq 1-C_{\eta_0}\cdot \eta_1^{3/2}\geq \tfrac{1}{2}\eeq
provided $\eta_1=\eta_1(\gamma)$ is chosen sufficiently small (as $\eta_0$, and hence $C_{\eta_0}$ is fixed).
The elements $\tilde \Om^j$ of $W_{\bar \eta}^{\perp}$ are now obtained from $\bar{\Om}^j$ as 
\beqs
\tilde \Om^j=\frac{\bar \Om^j-P_g^{W_{\bar\eta}}(\bar \Om^j)}{\norm{\bar \Om^j-P_g^{W_{\bar\eta}}(\bar \Om^j)}_{L^2(M,g)}}.\eeqs
As $b_0^j(P_g^{W_{\eta}}(\bar \Om^j))=0$, while of course $\norm{\bar \Om^j-P_g^{W_{\bar\eta}}(\bar \Om^j)}_{L^2(M,g)}\leq \norm{\bar \Om^j}_{L^2(M,g)}=1$ we hence have that $b_0^j(\tilde\Om^j)\geq b_0^j(\bar\Om^j)$ and the claim \eqref{est:b0-lower-half} follows from \eqref{snow1}.
\end{proof}

Combining this remark with \eqref{est:dz-L2-upper} and the explicit formula \eqref{def:tilde-th} for the principal part of $\tilde \Th^j$ on $\Col(\si^j)$, this remark implies in particular that 
\beqs 
\label{est:L2-tilde-theta-2}
 \Vert\tilde \Th^j\Vert_{L^2(M,g)}\leq C\ell_j\norm{dz^2}_{L^2(\Col(\si^j))}\leq C\ell_j^{-1/2}.
\eeqs
Using additionally that the principal and collar decay parts are $L^2$-orthogonal, compare \eqref{eq:orth-Four1}, we can obtain a far more refined bound on  $\norm{\tilde \Th^j}_{L^2}^2$ as a direct consequence of the above 
result from \cite{RT-neg}, while the expression for $\norm{dz^2}_{L^1}$ from \eqref{est:sizes_on_collars} furthermore gives a bound on the $L^1$-norm of $\tilde{\Th}^j$. To be more precise, we have

\begin{cor}\label{cor:RT-neg}
In the setting of Lemma~\ref{lemma:RT-neg} 
the elements $\tilde\Th^j$ characterised by \eqref{def:tilde-th} satisfy  
 \beqs 
 \label{est:L2-tilde-theta}
\big| \Vert\tilde \Th^j\Vert_{L^2(M,g)} -\tfrac{\ell_j}{\pi^2} \Vert dz^2\Vert_{L^2(\Col(\si^j),g)}\big|\leq C\ell_j^{5/2}
\text{ and }
\big|\norm{\tilde \Th^j}_{L^1(M,g)}-8\pi\big|\leq C\ell_j.
\eeqs
\end{cor}

\subsection{Proof of Proposition \ref{prop:RT-new} on the dual basis $\{ \Th^j\}_{j=1}^{3(\gamma-1)}$ to $\partial \ell_j$: }
\label{subsec:holo_om}
$ $

Let $(M,g)$ be a closed oriented surface of genus $\gamma$, let  $\mathcal{E}=\{\si^i\}_{i=1}^{3(\gamma-1)}$ be a decomposing collection of simple closed geodesics in a hyperbolic surface $(M,g)$, i.e. a collection of disjoint geodesics which decomposes $(M,g)$ into pairs of pants. Let $\{\Th^j\}_{j=1}^{3(\gamma-1)}$ be the basis of $\Hol(M,g)$ which is dual to 
 $\{\partial \ell_j\}_{j=1}^{3(\gamma-1)}$ of $\Hol(M,g)$.  In this section we want to derive the estimates on 
 $\Th^j$ and the corresponding renormalised elements 
$\Om^j$ stated in Proposition \ref{prop:RT-new}. The key step in this proof is to show the following Lemma~\ref{lemma:lower-est-princ-part} which allows us to bound the principal parts of $\Om^j$ on the corresponding collar. In the last part of the present section we will then combine this lemma with the results from \cite{RT-neg} that we recalled in Section~\ref{subsec:holo_main_1} to give the proof of  Proposition~\ref{prop:RT-new}. 

\begin{lemma}\label{lemma:lower-est-princ-part}
 Let $(M,g)$ be a closed oriented hyperbolic surface of genus $\gamma$ and let $\mathcal{E}=\{\si^1, \ldots, \si^{3(\gamma-1)}\}$ be a decomposing collection of disjoint  simple closed geodesics and let  $\eta\in (0,\arsinh(1))$ and  $\bar L<\infty$ be so that \eqref{ass:eta} and \eqref{ass:upperbound} are satisfied. Then for any number $\ell_0>0$ there exists a constant $\eps_0>0$ depending on $\ell_0$, $\eta$, $\bar L$ and the genus $\gamma$ such that
 the elements $\Om^j$ characterised by \eqref{def:Om-by-Th} satisfy
 \[ b_0 (\Omega^j, \Col(\si^j)) \geq \eps_0 \text{ for every } j \text{ for which } L_g(\si^j)\geq \ell_0.
 \]
\end{lemma}

We note that the sharp rate for the dependence of $\eps_0$ on $\ell_0$ is $\eps_0=C(\eta, \bar L,\gamma) \ell_0^{3/2}$, but that we shall not need this since the main purpose of the above lemma is to control the principal parts of elements $\Om^j$ corresponding to geodesics $\si^j$ which are not very short.

We note that as $\Om^j$ and $\Th^j$ are related by \eqref{eq:Om-Th} a lower bound on the principal part of $\Om^j$ is equivalent to an upper bound on the $L^2$-norm of the element $\Th^j$ of the dual basis, with the above result implying in particular that 
$$\norm{\Th^j}_{L^2}\leq C(\eta, \bar L, \gamma,\ell_0) \text{ for } j  \text{ so that } L(\si^j)\geq \ell_0.$$

\begin{rmk}\label{rmk:inverse}
If one would instead want to prove Proposition \ref{prop:RT-new} based on Wolpert's estimates on the gradient basis $\text{grad }\ell_j$ from \cite{Wolpert08, Wolpert12}, respectively its complex analogue $\th_{\si^j}$ considered in \cite{WII}, the main step in the proof would be to prove quantitative estimates on the \emph{inverse} of the matrix formed by the inner products $\langle \th_{\si^j},\th_{\si^k}\rangle$ or equivalently on the \emph{inverse} of the matrix $\langle \text{grad }\ell_j, \text{ grad } \ell_k \rangle$. These inner products are well controlled if at least one of the geodesics $\si^j$ and $\si^k$ is short, with the results of \cite{WII, Wolpert08, Wolpert12} yielding that $0<\langle \frac{\text{grad } \ell_j}{\norm{\text{grad } \ell_j}_{L^2}}, \frac{\text{grad } \ell_k}{\norm{\text{grad } \ell_k}_{L^2}}\rangle-\de_{ik}\leq C \ell_j^{3/2}\ell_k^{3/2}$. However this estimate does not yield sufficient information on the inner products of elements corresponding to geodesics whose length is of order $1$ to be able to derive quantitative bounds on the inverse matrix, as it would e.g. allow for two elements $\text{grad }\ell_{j,k} $ corresponding to two different, not too short, geodesics to be close to linear dependent. 
\\
We note that a bound on the inverse of this matrix is equivalent to a bound on the inner products of the elements of the dual bases. Thus also for this alternative route of proof based on results from \cite{Wolpert12} the main step would be to obtain suitable upper bounds on the norms of the elements of the dual spaces (in particular of those that correspond to geodesics that have length of order one) and hence to prove Lemma \ref{lemma:lower-est-princ-part}.

\end{rmk}

The proof of this lemma is based in particular on the following result from \cite{RTZ}: 

\begin{lemma}\label{lemma:W}[Contents of \cite[Lemma~2.4]{RTZ}]
 Let $(M,g_i)$ be a sequence of closed oriented hyperbolic surfaces that degenerate to a punctured (possibly disconnected) hyperbolic surface $(\Si, g_\infty)$ by collapsing $k$ geodesics $\si_i^j\subset (M,g_i)$, $j\in \{1, \ldots, k\}$ as described by the differential geometric version of the Deligne-Mumford compactness theorem, compare e.g. \cite[Proposition A.3]{RT-neg}.
Then
 \begin{align}\label{def:Wi}
  W_i\define \{ \Upsilon\in \Hol(M, g_i)\ |\ b_0(\Upsilon, \Col(\si_i^j))=0, \, j=1,\ldots, k\}
 \end{align}
converge to space $\Hol(\Si, g_\infty)$ of integrable holomorphic quadratic differentials on $(\Si,g_\infty)$
in the following sense:
\newline
There exists a sequence $\{w_i^j\}_{j=1}^{3(\gamma -1)-k}$ of orthonormal bases of $W_i$ and an orthonormal basis $\{w_\infty^j\}_{j=1}^{3(\gamma -1)-k}$ of $\Hol(\Si, g_\infty)$
so that for every $j\in\{1,,\ldots ,3(\gamma-1)-k\}$
$$f_i^*w_i^j\to w_\infty^j \text{ in }C_{loc}^\infty(\Si, g_\infty)\qquad \text{ as } i\to \infty,$$
where $f_i\colon \Si\to M\setminus \cup_{j=1}^{k} \si_i^j$ are the diffeomorphisms from the Deligne-Mumford compactness theorem for which $f_i^*g_i\to g_\infty$ in $C^\infty_{loc}(\Si, g_\infty)$.
\end{lemma}

The above lemma implies in particular:
\begin{cor}\label{cor:W}
In the setting of Lemma~\ref{lemma:W} the following holds true.
For any element $\Omega_\infty\in \Hol(\Si, g_\infty)$ there exists a sequence $\Om_i\in W_i\subset \Hol (M, g_i)$ with $$\Vert \Om_i\Vert_{L^2(M, g_i)}=\Vert \Om_\infty\Vert_{L^2(\Si, g_\infty)} \text{ so that }
f_i^*\Omega_i \to \Omega_\infty \text{ in } C_{loc}^\infty(\Si, g_\infty).$$
Conversely, any sequence of elements $\Om_i\in W_i\subset \Hol(M,g_i)$ with $\norm{\Om_i}_{L^2(M,g_i)}=1$ has a subsequence so that $f_i^*\Om_i$
 converges in $C^\infty_{loc}(\Si,g_\infty)$ to some $\Omega_\infty\in \Hol(\Si, g_\infty)$ with $\norm{\Omega_\infty}_{L^2(\Si,g_\infty)}=1$.
\end{cor}
 
 The proof of this corollary immediately follows by writing a given element  $\Omega_\infty\in \Hol(\Si, g_\infty)$ in the form $\Om_\infty= \sum_{j=1}^{3(\gamma -1)-k} a_jw_\infty^j$ and  considering the corresponding sequence $\Om_i =\sum_{j=1}^{3(\gamma -1) -k} a_jw_i^j$ respectively, for the second part of the Corollary, by writing a sequence $\Om_i$ in the form $\Om_i =\sum_{j=1}^{3(\gamma -1) -k} a_j^iw_i^j$ and passing to a subsequence for which the coefficients $(a_j^i)_i$ converge.

To prove Lemma~\ref{lemma:lower-est-princ-part} we will furthermore need 

\begin{lemma}\label{lemma:princ-part}
In the setting of Lemma~\ref{lemma:W} the following holds true. 
Let  $\si_\infty\subset (\Si,g_\infty)$ 
be a simple closed geodesic and let 
$\si_i$ be the (unique) simple closed geodesic in $(M,g_i)$ which is homotopic to $(f_i)_*\si_\infty\subset M$. 
Then for any sequence $\Om_i\in W_i$ for which  
\beq \label{eq:conv-ass-Om} 
f_i^* \Om_i\to \Om_\infty \text{ in } C^\infty_{loc}(\Si)
\qquad  \text{ as } i\to \infty \eeq
we have 
\beqs 
\label{eq:conv-princ=parts} 
b_0(\Om_i,\Col(\si_i))\to b_0(\Om_\infty, \Col(\si_\infty))\qquad  \text{ as } i\to \infty .\eeqs
\end{lemma}

\begin{proof}[Proof of Lemma~\ref{lemma:princ-part}]
We prove this lemma in two steps, first using the local smooth convergence of the metrics to obtain $C^1$-convergence of 
the geodesics $\tilde \si_i=f_i^*\si_i\subset (\Si,f_i^*g_i)$ to $\si_\infty$ and then in a second step using the relation \eqref{eq:length-evol} between the change of the length coordinate $d\ell(\Rea(\Om))$ and the principal part $\Rea(b_0(\Om))$ on the corresponding collar. We furthermore remark that it suffices to prove the convergence of the real parts $\Rea(b_0(\Om_i))$ as we may replace $\Om_i$ by $\i \Om_i$.

To begin with we note that there is a compact subset $K\subset \Si$
whose interior contains both the simple closed geodesic 
$\si_\infty \subset (\Si,g_\infty)$ as well as the 
simple closed geodesics $\tilde \si_i=f_i^*\si_i\subset (\Si,f_i^*g_i)$, for $i$ sufficiently large. Such a set can for example be obtained by setting $K=\de\thick(\Si,g_\infty)$
for
$\de<\half \arsinh(1)$ chosen so that 
the length of the shortest simple closed geodesic in $(\Si,g_\infty)$ is at least $4\de$. This choice of $\de$ ensures that 
$2\de\thin(\Si,g_\infty)$ is contained in the union of the 
collar neighbourhoods around the punctures, so the  
uniform convergence of the metrics on $K$ assures that for $i$ sufficiently large
$$\de\thin(\Si,g_\infty)\subset \bigcup_j f_i^* (\Col(\si_i^j)\setminus \si_i^j),$$
where $\Col(\si_i^j)$ are the collars around the collapsing geodesics $\si_i^j$ in $(M,g_i)$ and where we recall that $f_i$ is a diffeomorphism from $\Si$ to $ M\setminus(\bigcup_j \si_i^j)$.
Since the collar neighbourhoods of disjoint simple closed geodesics are disjoint, this ensures that $\tilde \si_i\subset \Col(\tilde \si_i)\subset \text{int}(K)$.

As the metrics converge smoothly in $K$ this allows us to obtain parametrisations $\gamma_i\colon S^1\to \tilde \si_i\subset (\Si,f_i^*g_i)$  (proportional to arc length) that converge in $C^1(\Si, g_\infty)$
 to a parametrisation $\gamma_\infty$ of $\si_\infty$ by a standard argument: convergence of the injectivity radii ensures that $\tilde \si_i\to \si_\infty$ in Hausdorff-distance (that may e.g. be computed w.r.t. $ g_\infty)$. Then using the geodesic equation we note that if the tangent vector $\gamma_i'(\th_1)$ were not close to the tangent vector at a nearby point on $\gamma_\infty$ then for some $c>0$ the point $\gamma_i(\th_1 +c)$ could not be close to $\si_\infty$. Appealing once more to the geodesic equation this then yields the desired $C^1$-convergence.

We now recall that 
along a curve of metrics $g(t)$ that evolves by 
 $\partial_t g(0)=\Rea \Om$ for some $\Om\in \Hol(M,g(0))$, the
evolution of the length $L_{g(t)}(\si(t))$ of the simple closed geodesic $\si(t)$ in $(M,g(t))$ that is  homotopic to $\si=\si(0)$  is determined by 
\eqref{eq:length-evol}.
As the geodesic $\si$ minimises length in its homotopy class, we have that
  \[ \frac{d}{dt}|_{t=0} L_{g(t)} (\si(t))= \frac{d}{dt} L_{g(t)} (\si(0)),\] 
  compare \cite[Rem.~4.11]{RT-neg}, so the principal part of a holomorphic quadratic differential $\Om\in \Hol(M,g=g(0))$ can be determined by 
  \beqs 
\label{eq:b0-by-dl}
\Rea(b_0(\Om,\Col(\si)))=-\frac{L_g(\si)}{2\pi^2}\frac{d}{dt}\vert_{t=0} L_{g+t\Rea(\Om)}(\si).
\eeqs
In the context of the lemma we thus find  that 
\beqa \label{est:diff-princ}
&\abs{\Rea\big(b_0(\Om_i, \Col(\si_i))-  b_0(\Om_\infty, \Col(\si_\infty))\big)}\\
&\qquad
=\babs{\frac1{2\pi^2}\big(  \ell_i \cdot 
\tfrac{d}{dt} L_{f_i^*(g_i +t\Rea (\Om_i))} (f_i^*\si_i)
 - \ell_\infty\cdot \tfrac{d}{dt} L_{g_\infty +t\Rea(\Om_\infty)} (\si_\infty)\big)}\\
  &\qquad \leq \,C \cdot |\ell_i-\ell_\infty|\cdot \abs{\tfrac{d}{dt}  L_{g_i +t\Rea (\Om_i)}(\si_i)} \\ 
   &\quad \qquad\, + C \left| \tfrac{d}{dt} \big( L_{f_i^*(g_i +t\Rea (\Om_i))} (f_i^*\si_i) - 
     L_{f_i^*(g_i +t\Rea (\Om_i))} (\si_\infty)\big)\right|\\
   &\quad \qquad \, + C \left| \tfrac{d}{dt} \big(  L_{f_i^*(g_i +t\Rea (\Om_i))} (\si_\infty) - 
     L_{g_\infty +t\Rea (\Om_\infty)} (\si_\infty)\big)\right|\\
     &\qquad =I+II+III
\eeqa where we remark that the real parts of holomorphic quadratic 
differentials need to be computed with respect to the corresponding conformal structure, where 
we write for short $\ell_i\define L_{g_i}(\si_i)\to \ell_\infty\define L_{g_\infty}(\si_\infty)>0$, 
and where derivatives with respect to $t$ are to be evaluated in $t=0$.
 
The first term is bounded by 
$$I\leq C \cdot |\ell_i-\ell_\infty| \cdot \frac{\abs{b_0(\Om_i,\Col(\si_i))}}{\ell_i} \leq C \cdot |\ell_i-\ell_\infty| \cdot\ell_i^{1/2}\cdot \norm{f_i^*\Om_i}_{L^2(K,f_i^*g_i)}
\to 0,$$
compare \eqref{est:b0-trivial-small}. 
    
To bound the third term in 
 \eqref{est:diff-princ} we remark that the convergence of the metrics implies that $\abs{\gamma_\infty'}_{f_i^*g_i}\to \abs{\gamma_\infty'}_{g_\infty}=\frac{\ell_\infty}{2\pi}>0$. For $i$ sufficiently large we may thus use \eqref{eq:conv-ass-Om}  and the convergence of the conformal structures to conclude that 
  \beqas 
 III
 =&\ C \bigg| \int_{S^1} \frac{ \Rea (f_i^*\Om_i)(\gamma_\infty',\gamma_\infty')} 
 {\abs{\gamma_\infty'}_{f_i^*g_i}}
  -\frac{\Rea (\Om_\infty)(\gamma_\infty',\gamma_\infty') }{\abs{\gamma_\infty'}_{g_\infty}}
  \bigg| \\
  \leq &\ C\norm{\Rea (f_i^*\Om_i)-\Rea (\Om_\infty)}_{L^\infty(K)}+C\norm{f_i^*g_i-g_\infty}_{L^\infty(K)}\cdot \norm{f_i^*\Om_i}_{L^\infty(K)}\to 0
 \eeqas
 as $i\to \infty$. 
Similarly, we may bound the second term in \eqref{est:diff-princ} by 
$$II\leq C\cdot \norm{\gamma_i- \gamma_\infty}_{C^1(K)}\cdot \norm{\Rea (f_i^*\Om_i)}_{C^1(K)}\to 0$$
where we may compute the norms with respect to any of the equivalent metrics, say w.r.t. $g_\infty$. Combined we thus obtain the claim of the lemma.
\end{proof}
Based on the above results we can now give the
\begin{proof}[Proof of Lemma~\ref{lemma:lower-est-princ-part}]
 We argue by contradiction. So let us assume that there exist positive numbers $\bar L$, $\eta$ and $\ell_0$ and a sequence of closed hyperbolic surfaces $(M, g_i)$ together with decomposing  collections $\mathcal{E}_i=\{ \si^1_i, \ldots, \si^{3(\gamma-1)}_i\}$ of simple closed geodesics which satisfy \eqref{ass:eta} and \eqref{ass:upperbound} 
for $\bar L$ and $\eta$ so that, after reordering the geodesics if necessary,  
 $$L_{g_i}(\si^{3(\gamma-1)}_i)\geq \ell_0 \text{ but } b_0 (\Om^{3(\gamma-1)}_i, \Col(\si^{3(\gamma-1)}_i))\to 0$$ 
 for the renormalised elements $\Om^j_i=-\Th^j_i \Vert \Th^j_i\Vert_{L^2(M,g_i)}^{-1}$  of the dual bases $\{\Th_i^j\}_{j=1}^{3(\gamma-1)}$ of $\Hol(M,g_i)$ corresponding to $\mathcal{E}_i$, compare \eqref{def:Om-by-Th}. 

After passing to a subsequence, and if necessary relabelling the geodesics $\{\si_i^j\}_{j=1}^{3(\gamma-1)-1}$, we may assume by the Deligne-Mumford compactness theorem, see e.g. \cite[Prop.~5.1]{Hu} or \cite[Prop.~A.3]{RT-neg}, that $(M, g_i)$ converges to a (possibly punctured and disconnected) hyperbolic surface by collapsing the $k\in \{0,\ldots, 3(\gamma -1) -1\}$ geodesics $\si_i^j$, $j=1,\ldots, k$. We note that in the above description only geodesics that are contained in $\mathcal{E}_i$ can collapse since each $\mathcal{E}_i$ satisfies assumption \eqref{ass:eta} for some fixed $\eta>0$.
 
 By construction $b_0 (\Om^{3(\gamma -1)}_i, \Col(\si_i^{j}))= 0$ for
 all $j<3(\gamma-1)$, in particular for $j=1,\ldots, k$,  so $\Om^{3(\gamma -1)}_i$ is an element of the space $W_i$ defined in \eqref{def:Wi}. By Corollary \ref{cor:W} we may thus pass to a
subsequence to obtain that 
$$f_i^*\Om_i^{3(\gamma -1)}\to \Om_\infty \text{ in } C_{loc}^\infty(\Si, g_\infty) \text{ for some } \Om_\infty\in \Hol(\Si, g_\infty) \text{ with } \Vert\Om_\infty\Vert_{L^2(\Si, g_\infty)}=1.$$ 
 We recall that since the simple closed geodesics $\si_i^{k+1}, \ldots, \si_i^{3(\gamma -1)}$ are disjoint from the collars $\Col(\si_i^j)$, $j=1,\ldots, k$, we may choose a compact subset $K\subset \Si$
 as in the proof of Lemma~\ref{lemma:princ-part}
so that for $i$ sufficiently large
$$ \tilde \si_i^j\define f_i^*\si_i^j \subset K \text{ for every } j=k+1,\ldots, 3(\gamma -1).$$
Since the metrics $f_i^*g_i$ converge smoothly to $g_\infty$ on $K$ 
we obtain from \eqref{ass:upperbound}
that  for $i$ sufficiently large
also
 $L_{g_\infty}(\tilde \si_i^j)\leq \bar L+1$ for each $k+1\leq j\leq 3(\gamma-1)$. 
As there are only finitely many homotopy classes of closed curves in $(\Si,g_\infty)$ which have a representative of length no more than $\bar L+1$, we may thus pass to a further subsequence in a way that ensures that for each $j=k+1, \ldots, 3(\gamma-1)$  the curves $\tilde\si_i^j$, $i\in \N$, are homotopic to each other. 
We denote the simple closed geodesic in $(\Si, g_\infty)$ that belongs to this homotopy class $[\tilde \si^j_i]$ by $\si_\infty^j$ and remark that $\{\si^j_\infty\}_{j=k+1}^{3(\gamma-1)}$ decomposes  $(\Si, g_\infty)$ into pairs of pants. 

We recall that $b_0(\Om^{3(\gamma-1)}_i, \Col(\si^{3(\gamma-1)}_i))\to 0$ while by definition $b_0(\Om^{3(\gamma-1)}_i, \Col(\si^{j}_i))=0$ for $j\leq 3(\gamma-1)-1$. Hence Lemma~\ref{lemma:princ-part} implies that 
\beq\label{eq:princ-parts-zero}
b_0(\Om_\infty, \Col(\si^{j}_\infty))=\lim_{i\to \infty} b_0(\Om^{3(\gamma-1)}_i, \Col(\si^{j}_i))=0\eeq
 for the whole decomposing collection $\{\si^j_\infty\}_{j=k+1}^{3(\gamma-1)}$ of $(\Si, g_\infty)$. 
 
However, as the map $\Upsilon\mapsto (\partial \ell_{k+1},\ldots, \partial \ell_{3(\gamma-1)})$ is an isomorphism from $\Hol(\Si,g_\infty)$ to $\C^{3(\gamma-1)-k}$, compare \cite[Theorem~3.7]{Wolpert82} and Remark \ref{rmk:extra-iso}, while \eqref{eq:dell-c} and \eqref{eq:princ-parts-zero} imply that the image of $\Om_\infty$ under this map is zero, we obtain that  $\Omega_\infty\equiv 0$ in contradiction to 
$\Vert \Omega_\infty\Vert_{L^2(\Si, g_\infty)}=1$.
\end{proof} 
In order to prove Proposition \ref{prop:RT-new}, we now want to use Lemma \ref{lemma:lower-est-princ-part} to relate the elements of the full dual basis $\{\Om^j\}$ of $\Hol(M,g)$ to the basis $\{\tilde \Om^j\}$ of $\ker(\partial\ell_1,\ldots,\partial \ell_k)$ for which the results of \cite{RT-neg} that we recalled in Lemma \ref{lemma:RT-neg} already establish precisely the type of bounds that we wish to prove for $\Om^j$. 
\begin{lemma}\label{lemma:Om} 
Let $(M,g)$, $\mathcal{E}$ and $\eta$ be as in Proposition~\ref{prop:RT-new}, where we can assume that $\mathcal{E}$ is ordered so that the simple closed geodesics of length no more than $2\bar \eta\define 2\min(\eta, \eta_1)$, $\eta_1$ as in Lemma~\ref{lemma:RT-neg} and Remark \ref{rmk:half}, are given by 
 $\si^1,\ldots, \si^k$,  $k\in \{0,\ldots, 3(\gamma-1)\}$. Let furthermore 
$\{\tilde\Om^j\}_{j=1}^k$ respectively  $\{\tilde \Th^j\}_{j=1}^k$  be the bases of $\ker(\partial \ell_1,\ldots, \partial \ell_k)^\perp$ from Lemma~\ref{lemma:RT-neg} respectively Corollary~\ref{cor:RT-neg} and let $\{\Om^j\}_{j=1}^{3(\gamma-1)}$ and  $\{ \Th^j\}_{j=1}^{3(\gamma-1)}$ be the bases of $\Hol(M,g)$ characterised by \eqref{def:Om-by-Th}.
 \newline
Then the following claims hold true for a constant 
 $C$ that depends only on the genus and the numbers $\eta$, $\bar L$ from \eqref{ass:eta} and \eqref{ass:upperbound}:
 \newline
For every $j=1,\ldots, 3(\gamma-1)$
  \beq 
 \label{est:Th-L2-upper} 
 \norm{\Th^j}_{L^2(M,g)}\leq C\ell_j^{-1/2}
 \eeq
while for $j=1,\ldots, k$ furthermore $\Th^j\sim \tilde \Th^j$ and $\Om^j\sim\tilde  \Om^j$ in the sense that 
\beq
\label{eq:Th-by-tilde-Th}
\Th^j= \tilde\Th^j+v^j \text{ for some } 
 v^j\in \ker(\partial \ell_1,\ldots, \partial \ell_k)
  \quad \text{ with }\norm{v^j}_{L^\infty(M,g)}\leq C \ell_j,
 \eeq 
respectively 
\beq
\label{eq:Om-by-tilde-Om}
\Om^j= a_j \tilde\Om^j+w^j \text{ for some } a_j\in \R^+ \text{ and } w^j\in \ker(\partial \ell_1,\ldots, \partial \ell_k)
\eeq
for which
\beq
\label{est:coeff_Om-by-tilde-Om}
\abs{1-a_j}\leq C\ell_j^3 \text{ while } \norm{w^j}_{L^\infty(M,g)}\leq C\ell_j^{3/2}.
\eeq
 \end{lemma}

We note that  the upper bound \eqref{est:Th-L2-upper} 
 on the dual basis is equivalent to a lower bound \eqref{est:lower-hallo} on the principal part of the $\Om^j$, since $\Th^j$ and $\Om^j$ are related by 
\beq\label{eq:Om-Th}
\Th^j=-\frac{\ell_j}{\pi^2 b_0(\Om^j, \Col(\si^j))} \Om^j, 
\eeq
as $b_0^j(\Th^j)=\frac{\ell_j}{\pi^2}$, compare \eqref{eq:dell-c}. We also remark that proving such a lower bound on $b_0^j(\Om^j)$ can be seen to be equivalent to establishing bounds on the inverse of the isomorphism 
\beq
\label{def:iso}
\Hol(M,g)\ni \Upsilon\mapsto \left(\partial \ell_1(\Upsilon), \ldots, \partial \ell_{3(\gamma-1)}(\Upsilon)\right)\in \C^{3(\gamma-1)}.
\eeq

\begin{proof}[Proof of Lemma~\ref{lemma:Om}] \label{proof:Om}
We will first prove 
the claims \eqref{eq:Om-by-tilde-Om} and \eqref{est:coeff_Om-by-tilde-Om} on $\Om^j$, $j=1,\ldots, k$, then establish that \eqref{est:Th-L2-upper} holds true (for any $j$) and finally combine these two parts of the lemma to derive 
\eqref{eq:Th-by-tilde-Th}. 

So let 
$j\in \{1,\ldots, k\}$ and hence $\ell_j\leq 2 \bar \eta$, which allows us to apply the results of \cite{RT-neg} on the corresponding elements $\tilde \Om^j$ and $\tilde \Th^j$ which we recalled in Lemma~\ref{lemma:RT-neg} and Corollary \ref{cor:RT-neg}, as well as Remark \ref{rmk:half}.
 As $\{\Om^j\}_{j=1}^{3(\gamma-1)}$ is a basis of $\Hol(M,g)$, 
we can write each such 
$\tilde \Om^j=\sum_{i=1}^{3(\gamma-1)} d^j_i \Om^i$ for complex coefficients $d^j_i$ which we claim must be so that 
$d_j^j\in \R^+$ while $d_j^i=0 $ for $i\in\{1,\ldots, k\}$ with $i\neq j$: Indeed the first property follows 
since  the principal parts of  $\Om^j$ and $\tilde \Om^j$ on $\Col(\si^j)$ are both positive while 
$b_0(\Om^i,\Col(\si^j))=0$ for every
$i\neq j$, while the second property follows since for $i\in\{1,\ldots, k\}$ with $i\neq j$
 the only element in the above expression whose principal part on $\Col(\si^i)$ is non-zero is  $\Om^i$.

We may thus write each $\Om^j$, $j=1,\ldots,k$, in the form 
\beq \label{eq:proof-Om}  \Om^j=a_j \cdot \big[\tilde{\Om}^j + \sum_{m=k+1}^{3(\gamma-1)} c_{m}^j\Om^m \big] \text{ for some } a_j\in \R^+ \text{ and } c_{m}^j\in\C,
  \eeq
 and we claim that 
  \beq 
\label{claim:coeff-lemma-Om-proof}
\abs{1-a_j}\leq C\ell_j^3 \text{ while } 
 |c_{m}^j|\leq C\ell_j^{3/2},\quad  m\in\{k+1,\ldots, 3(\gamma-1) \}.
\eeq
As $\Om^m\in \ker(\partial \ell_1,\ldots,\partial \ell_k)$, $m \geq k+1$, and hence, by \eqref{est:W2}, $\norm{\Om^m}_{L^\infty(M,g)}\leq C_\eta \norm{\Om^m}_{L^1(M,g)}\leq C_\eta\norm{\Om^m}_{L^2(M,g)} \Area(M,g)^{1/2} \leq C(\eta,\gamma)$, 
 this will imply the two claims \eqref{eq:Om-by-tilde-Om} and \eqref{est:coeff_Om-by-tilde-Om} about $\Om^j$ made in the lemma.

To prove the claimed estimate on $c_{m}^j$ we compare the principal parts in \eqref{eq:proof-Om} on $\Col(\si^m)$, $m\geq k+1$: since $b_0(\Om^i,\Col(\si^m))=0$ for $i\neq m$,
 while, by Lemma \ref{lemma:lower-est-princ-part}, 
 $b_0(\Om^m,\Col(\si^m))\geq \eps_0(\bar \eta, \gamma, \bar L)>0$,
we may use \eqref{est:b0-trivial-2}
to
bound 
\beqas
 \abs{c_{m}^j}&=\babs{\frac{ b_0(\tilde{\Om}^j, \Col( \si^m))}{b_0(\Om^m, \Col( \si^m))}}\leq C |b_0(\tilde{\Om}^j, \Col(\si^m))|
\leq C \Vert \tilde{\Om}^j\Vert_{L^2(\Col(\si^m))}\leq C\Vert \tilde{\Om}^j\Vert_{L^2(M\setminus \Col(\si^j))}
\leq  C\ell_j^{3/2},
\eeqas
where the fact that  collars around disjoint geodesics are disjoint is used in the penultimate step (recall that $j<m$), while estimate \eqref{est:RT-neg1} of Lemma~\ref{lemma:RT-neg} is used in the last step.

Having thus established the bound on $c_{m}^j$ claimed in \eqref{claim:coeff-lemma-Om-proof} we now turn to the analysis of $a_j$ which is characterised by 
$a_j=\frac{b_0(\Om^j,\Col(\si^j))}{b_0(\tilde \Om^j,\Col(\si^j))}\in\R^+.$
Using \eqref{est:b0-lower-half} we obtain an initial 
bound of   
$$a_j=\frac{b_0(\Om^j,\Col(\si^j))}{b_0(\tilde \Om^j,\Col(\si^j))}\leq 2 b_0(\Om^j,\Col(\si^j))\norm{dz^2}_{L^2(\Col(\si^j))}\leq 2\norm{\Om^j}_{L^2(M,g)}=2. $$
To obtain the more precise bound on $a_j$ claimed in \eqref{claim:coeff-lemma-Om-proof} we first remark that 
 the elements  $\Om^m$, $m\geq k+1$, are almost orthogonal to 
 $\tilde \Om^j$. To be more precise, since $b_0(\Om^m,\Col(\si^j))=0$ while $\tilde \Om^j$ is essentially given by $b_0(\tilde \Om^j,\Col(\si^j))dz^2$, compare \eqref{est:RT-neg1}, we may 
use \eqref{eq:orth-Four1} to write
\beqas 
 \abs{ \langle \tilde{\Om}^j, {\Om}^m\rangle_{L^2(M,g)} } &=   \abs{\langle \tilde{\Om}^j - b_0(\tilde{\Om}^j, \Col(\si^j))dz^2 , {\Om}^m\rangle_{L^2(\Col(\si^j))} +   \langle \tilde{\Om}^j, {\Om}^m\rangle_{L^2(M\setminus \Col(\si^j))}}\\
  &\leq  C \Vert  \tilde{\Om}^j - b_0(\tilde{\Om}^j, \Col(\si^j))dz^2\Vert_{L^\infty(\Col(\si^j))}  + C \Vert  \tilde{\Om}^j \Vert_{L^\infty(M\setminus \Col(\si^j))}\\
& \leq C\ell_j^{3/2}.
  \eeqas
Comparing the norms of the two sides of \eqref{eq:proof-Om}, we may hence conclude that indeed
\beqas
\abs{1-a_j^2}&=\babs{\norm{\Om^j}_{L^2(M,g)}^2-\norm{a_j\tilde \Om^j}_{L^2(M,g)}^2 }\\
&\leq 2|a_j|\sum_m 
\abs{c_{m}^j}\abs{  \langle \tilde{\Om}^j, {\Om}^m\rangle_{L^2(M,g)} }+|a_j|^2\norm{ \sum_m c_{m}^j\Om^m}_{L^2(M,g)}^2
\leq C\ell_j^3,
\eeqas
as claimed, which completes the proof of the claims \eqref{eq:Om-by-tilde-Om} and \eqref{est:coeff_Om-by-tilde-Om} on $\Om^j$. 

We now turn to the proof of \eqref{est:Th-L2-upper}, which we recall is the only estimate that we need to prove for all $j\in\{1,\ldots,3(\gamma-1)\}$ rather then just for those $j$ for which $\ell_j\leq 2\bar\eta$. To this end we first choose $\ell_0\in (0,2\bar \eta)$ small enough so that \eqref{claim:coeff-lemma-Om-proof} ensures that if 
$\ell_j\leq \ell_0$ (and hence in particular $j\leq k$) then $a_j\geq \half$. For such indices we hence obtain from Remark \ref{rmk:half} and Lemma \ref{lemma:RT-neg} that
\beq \label{est:lower-hallo}
b_0^j(\Om^j)\norm{dz^2}_{L^2(\Col(\si^j))}\geq\epsilon_1
\eeq
holds true for any $\eps_1\leq \frac14$, while Lemma \ref{lemma:lower-est-princ-part} ensures that this bound holds true for any other $j\in\{1,\ldots,3(\gamma-1)\}$ for a constant $\eps_1$ that depends only on the genus, $\eta$ and $\bar L$. 

As $\Th^j$ is related to $\Om^j$ by \eqref{eq:Om-Th} we hence obtain the desired upper bound on $\norm{\Th^j}_{L^2}$ from 
\beqas  \norm{\Th^j}_{L^2(M,g)}  &\leq \frac{\ell_j\cdot \norm{dz^2}_{L^2(\Col(\si^j))}}{\pi^2 b_0( \Om^j, \Col(\si^j))\cdot \norm{dz^2}_{L^2(\Col(\si^j))}} \leq C\ell_j  \norm{dz^2}_{L^2(\Col(\si^j))}\leq C\ell_j^{-1/2}.\eeqas

To complete the proof of the lemma it now remains to show that 
$\Th^j=-\norm{\Th^j}_{L^2}\Om^j$, $j\leq k$, is described by \eqref{eq:Th-by-tilde-Th}. Multiplying the expression \eqref{eq:Om-by-tilde-Om} for $\Om^j$ that we have already proven with $-\norm{\Th^j}_{L^2}$, and using  that $\tilde \Om^j$ is a real multiple of $\tilde \Th^j$, allows us to write $\Th^j= c_j \tilde \Th^j +v_j$ for some $c_j\in\R$, later shown to be $c_j=1$,  and $v_j=-\norm{\Th^j}_{L^2}w_j$. The claimed bound on $\norm{w^j}_{L^\infty}$ hence follows from \eqref{est:coeff_Om-by-tilde-Om} and the fact that $\norm{\Th^j}_{L^2}\leq C\ell_j^{-\half}$. We finally recall that by definition 
$\partial \ell_j(\Th^j)=\partial \ell_j(\tilde \Th^j)=1$ and thus, by \eqref{eq:dell-c}, 
$b_0(\tilde \Th^j, \Col(\si^j))=-\frac{\ell_j}{\pi^2}= b_0(\Th^j, \Col(\si^j))$, while 
$v^j\in \ker ( \ell_j)$ and thus $b_0(v^j, \Col(\si^j))=0$. Hence the leading coefficient $c_j$ must
indeed be identically $1$.
\end{proof}

 Proposition~\ref{prop:RT-new} now follows by a short argument that combines the results from \cite{RT-neg} on $\tilde \Om^j$ that we recalled in Lemma~\ref{lemma:RT-neg} with Lemma~\ref{lemma:Om}.
\begin{proof}[Proof of Proposition~\ref{prop:RT-new}]
To begin with we recall that the upper bound in \eqref{est:RT-neg3-new} is trivially satisfied as $\norm{\Om^j}_{L^2(M,g)}=1$, compare \eqref{est:b0-trivial}, and that we have already established \eqref{est:RT-neg3-newest} in the above proof of Lemma \ref{lemma:Om}, compare \eqref{est:lower-hallo}.

In case that $\ell_j\geq 2\bar\eta$, for $\bar \eta$ as in Lemma \ref{lemma:Om},
the lower bound on the principal part in \eqref{est:RT-neg3-new}
is trivially satisfied if $C$ is chosen sufficiently large and also \eqref{est:RT-neg1-new} is trivially satisfied thanks to \eqref{est:W2}. 

On the other hand, for indices with $\ell_j\leq  2\bar\eta$, the claim  \eqref{est:RT-neg1-new} 
is a direct consequence of the corresponding bound \eqref{est:RT-neg1} for $\tilde \Om^j$ and 
the relations \eqref{eq:Om-by-tilde-Om} and \eqref{est:coeff_Om-by-tilde-Om} between $\Om^j$ and $\tilde \Om^j$ and the lower bound of \eqref{est:RT-neg3-new} follows from \eqref{est:RT-neg3}  as
$ b_0(\Om^j, \Col(\si^j))=a_jb_0(\tilde\Om^j, \Col(\si^j))$ with $a_j$ satifying \eqref{est:coeff_Om-by-tilde-Om}.

Finally, to estimate  the inner products, 
we combine \eqref{eq:orth-Four1} with the bounds from \eqref{est:RT-neg1-new} 
to obtain that for $i\neq j$
\begin{align*}
 |\langle \Om^i, \Om^j\rangle_{L^2(M,g)}|\leq&\  \Vert \Om^i-b_0(\Om^i, \Col(\si^i))dz^2\Vert_{L^\infty(\Col(\si^i))} \Vert \Om^j\Vert_{L^1(\Col(\si^i))}  \\
&\ + \Vert \Om^j-b_0(\Om^j, \Col(\si^j))dz^2\Vert_{L^\infty(\Col(\si^j))} \Vert \Om^i\Vert_{L^1(\Col(\si^j))}  \\
 &\ + \Vert \Om^i\Vert_{L^2(M\setminus \Col(\si^i))}\Vert \Om^j\Vert_{L^2(M\setminus \Col(\si^j))}\\
 \leq & C\ell_i^{3/2} \ell_j^{3/2}.\qedhere
\end{align*}
\end{proof}
From Proposition \ref{prop:RT-new} we also obtain the following useful estimates for the bases $\{\Om^j\}$ and $\{\Th^j\}$: 
\begin{rmk}\label{rmk:lower_lin_comb}
The uniform lower bound on the principal parts \eqref{est:RT-neg3-newest}, together with \eqref{eq:orth-Four1}, implies in particular that 
the estimate
\beqs \norm{\sum c_j\Om^j}_{L^2(M,g)}^2\geq \eps_1^2\cdot  \sum \abs{c_j}^2  \text{ for all } c_j\in\C \eeqs
holds true for the same constant $\eps_1>0$ for which also \eqref{est:RT-neg3-newest} holds true.
Combining 
\eqref{est:RT-neg3-new} and \eqref{est:RT-neg3-newest} with \eqref{est:sizes_on_collars}, \eqref{est:dz-lower} and \eqref{est:dz-L2-upper}  furthermore implies that 
\beq
\label{est:princ-parts-Om-Th}
c\ell_j^{3/2} \leq b_0^j(\Om^j)\leq C\ell_j^{3/2} 
\eeq
for some $c=c(\eta,\bar L,\gamma)>0$ and $C=C(\bar L)$. 
Finally we remark that the analogue of 
the estimates from Corollary~\ref{cor:RT-neg} remain valid for the $\Th^j$, namely, 
combining \eqref{est:princ-parts-Om-Th} and  \eqref{est:RT-neg1-new} with \eqref{eq:orth-Four1} gives
$\abs{\norm{\Th^j}_{L^2(M,g)}^2- \frac{\ell_j^2}{\pi^4}\norm{dz^2}_{L^2(\Col(\si^j))}^2}\leq C\ell_j^2$, which, when combined with  the expression for $\norm{dz^2}_{L^2}$ from \eqref{est:sizes_on_collars}, yields
\beq
\label{est:precise-norm-Th}
\abs{\norm{\Th^j}_{L^2(M,g)}^2-\tfrac{32\pi}{\ell_j}}\leq C\ell_j^2.
\eeq
\end{rmk}

In addition, it is useful to observe that for the elements $\tilde \Om^j$ considered in \cite{RT-neg} we obtain: 
\begin{rmk}\label{rmk:extending}
As we may extend 
an arbitrary 
collection $\{\si^1,\ldots, \si^k\}$ of disjoint simple closed geodesics  to a collection which decomposes $(M,g)$ into pairs of pants which satisfies \eqref{ass:upperbound} for some 
$\bar L=\bar L(\gamma, \max_{j=1,\ldots,k}(L_g(\si^j)))$, compare \cite[Theorem~3.7]{Hu}, it is now easy to prove that 
 Lemma~\ref{lemma:RT-neg} remains valid also without the  smallness assumption on the $\ell_j$'s and with constants that depend only on the numbers $\bar L$ and $\eta$ for which \eqref{ass:eta} and \eqref{ass:upperbound} are satisfied and as usual the genus $\gamma$. 
 In particular, for any such collection a uniform lower bound of 
\beq\label{est:b0-tilde-lower}
b_0^j(\tilde\Om^j)\norm{dz^2}_{L^2(\Col(\si^j))} \geq \tilde\eps_1
\eeq
holds true for a constant $\tilde \eps_1$ that depends only on $\gamma$, $\eta$ and $\bar L$.
\end{rmk}

This final remark is easily derived by combining the estimates on the full dual basis $\Om^j$ of $\Hol(M,g)$ corresponding to the above extended set of geodesics with the fact that $
\tilde \Om^j=\frac{\Om^j-P_g^{W}( \Om^j)}{\norm{ \Om^j-P_g^{W}( \Om^j)}_{L^2(M,g)}}$, as elements of $W=\ker(\partial \ell_1,\ldots, \partial \ell_k)$ have zero principal part on $\Col(\si^j)$, which, combined with \eqref{est:RT-neg3-newest}, yields $\eps_1\leq b_0^j(\Om^j)\norm{dz^2}_{L^2(\Col(\si^j))}\leq 
\norm{\Om^j-P_g^{W}( \Om^j)}_{L^2(M,g)}\leq 1$.

\subsection{Proof of Theorem \ref{thm:Dehn-twists-main} on the elements of $\Hol$ representing Dehn-twist}
\label{subsec:holo_twist}
$ $

Before we begin to prove our main result, Theorem \ref{thm:Dehn-twists-main}, on the elements of  $\Hol$  which generate only Dehn-twists we note that a 
\textit{non-horizontal} curve of hyperbolic metrics $g(t)$ which moves only by Dehn-twists around a given geodesic $\si^j$ can be explicitly be constructed as follows:
Let $\Col(\si^j)$ be the collar region around $\si^j$ which we recall is disjoint from the collars around all other geodesics which are disjoint from $\si^j$. We can then cut the collar off from the surface, twist it by angle $t$ and glue it back into the surface using the same gluing map.

The twisting on the collar can be obtained by pulling back the metric on this collar (not on the whole surface)
 with a diffeomorphism that is given in collar coordinates, compare Lemma~\ref{lemma:collar}, by
$$f_t\colon \Col(\si^j)\to \Col(\si^j),\ f_t(s,\th)=(s,\th+t\xi(s))$$
where we choose $\xi$ such that $\xi\equiv \pm \frac12$ 
for $s$ near $\pm X(\ell_j)$.  Thus near the ends of the collars
we just carry out a rotation by a fixed angle $\pm \frac12t$ so we can glue the collar back to the rest of the surface and thus obtain a metric whose twist coordinate $\psi^j$ has increased by $t$. This results in a smooth curve of complete hyperbolic metrics $(g(t))_t$ with $g(0)=g$ which evolves by 
$$\pt g(0)=k^j=
\begin{cases} 0 &\text{ on } M\setminus \Col(\si^j)\\
\xi'(s)\rho^2(s)(ds\otimes d\th+d\th\otimes ds)& \text{ on }  \Col(\si^j).\end{cases}
 $$
We note that we can equivalently view $k^j$ as the real part of the quadratic (but not holomorphic) differential $K^j$ which is supported on $\Col(\si^j)$
and there given in collar coordinates by 
$$ K^j=-\i \xi'(s)\rho^2(s) \, (ds+\i d\th)^2.$$
We now recall that neither the length nor the twist coordinates of a metric can be changed by pulling back the metric with a diffeomorphism that is homotopic to the identity (and of course defined on all of $M$ as opposed to the $f_t$ above). 
Hence the horizontal part 
$P_g^\Hol(K^j)$
of $\pt g(0)=\Rea(K^j)=L_Xg+\Rea(\Psi^j) $ induces the same Dehn-twist as $K^j$, so we can characterise the element $\Psi^j$ of $\Hol(M,g)$ that we analyse in the present section by 
$$\Psi^j=P_g^\Hol(K^j).$$

We note that since $\Rea(\Psi^j)$ leaves the length coordinates invariant, we have,  by \eqref{eq:length-evol}, that
\beq 
\label{eq:other-princ-part-Psi}
\Rea(b_0(\Psi^j,\Col(\si^i)))=0 \text{ for every } i=1,\ldots, 3(\gamma-1),
\eeq

while the above expression for $\Psi^j$ furthermore implies that for every $\Upsilon\in \Hol(M,g)$ 
\begin{align}
\langle \Upsilon,\Psi^j\rangle&=
\langle \Upsilon, K^j\rangle =\i b_0(\Upsilon, \Col(\si^j))\int_{\Col(\si^j)} \rho^2 \xi'\abs{dz^2}_g^2 dv_g
= 8\pi \i b_0(\Upsilon, \Col(\si^j)) \label{eq:inner-prod-Psi}\end{align}
since the other Fourier modes are orthogonal to $dz^2$ on every circle $\{s\}\times S^1$, in particular $\Psi^j \perp \ker(\partial \ell_j)$. 
We note that these  properties of $\Psi^j$ could alternatively be derived based on the length-twist duality of Wolpert \cite[Section~3]{Wolpert82}.

\begin{rmk} \label{rmk:extra-iso}
We note that the above expression and properties of the elements inducing Dehn-twists allow for a short proof of the fact that the map \eqref{def:iso} is an isomorphism, both in the setting of closed hyperbolic surfaces as considered in \cite[Thm.~3.7]{Wolpert82} and also for punctured hyperbolic surfaces as  considered in the proof of Lemma \ref{lemma:lower-est-princ-part} above.
\end{rmk}

 \begin{proof}[Proof of Theorem~\ref{thm:Dehn-twists-main}]
As $\Psi^j\perp \ker(\partial \ell_j)=\text{span}\{\Om^i\}_{i\neq j}$, 
we may write 
\beqa\label{eq:ren_psi_proof}
 \frac{\Psi^j}{\Vert \Psi^j\Vert_{L^2(M,g)}} = -a_j\cdot \i (\Om^j - P_g^{\ker\partial\ell_j} (\Om^j)),
  \eeqa 
  where we note that $a_j\in \R^+$
  since $ b_0^j(\Psi^j)=b_0(\Psi^j, \Col(\si^j))\in \i\R_-$, compare \eqref{eq:inner-prod-Psi}, while $b_0^j(\Upsilon)=0$ for  $\Ups\in \ker(\partial \ell_j)$. 
Clearly also $\abs{a_j}\geq 1$ as the element on the left hand side has norm $1\geq \norm{\Om^j - P_g^{\ker\partial\ell_j} (\Om^j)}_{L^2(M,g)} $.

The trivial upper bound \eqref{est:b0-trivial} on $b_0^j(\frac{\Psi^j}{\norm{\Psi^j}{L^2}})$ 
yields 
that 
 $$\abs{ a_j\cdot b_0^j(\Om^j)} \cdot  \norm{dz^2}_{L^2(\Col(\si^j))}= \abs{b_0^j\big(\tfrac{\Psi^j}{\Vert \Psi^j\Vert_{L^2(M,g)}}\big) }\cdot \norm{dz^2}_{L^2(\Col(\si^j))} \leq 1$$
and hence, thanks to the estimates \eqref{est:RT-neg3-new} and \eqref{est:RT-neg3-newest} for the principal part of $\Om^j$,
$$1\leq a_j\leq \frac{1}{\max(\eps_1,1-C\ell_j^{3})}\leq 1+C\ell_j^3,$$
resulting in the second claim of \eqref{est:coeff-Psi}.

We then note that since the principal part of $\Psi^j$ on each of the collars $\Col(\si^k)$ is purely imaginary and since 
 $\{\Om^k\}_{k\neq j}$ is a basis of $\ker(\partial \ell_j)$ we may  write the second term in  \eqref{eq:ren_psi_proof}  as
\beq \label{eq:writ_second_part_Psi}
a_j\i P_g^{\ker\partial \ell_j} (\Om^j)= \sum_{k\neq j} c_k^j \i\Om^k, \eeq
for coefficients $c_k^j\in \mathbb R$, which gives the claimed expression \eqref{eq:write-Psi-with-Om} for $\tfrac{\Psi^j}{\norm{\Psi^j}_{L^2}}$. 
 
 In order to prove the  estimates on the coefficients $c_k^j$ claimed in \eqref{est:coeff-Psi}
 we first prove that 
 $\norm{P_g^{\ker\partial \ell_j} (\Om^j)}_{L^2}\leq C\ell_j^{3/2}$, which, by  Remark~\ref{rmk:lower_lin_comb},
will give an initial bound of  $\abs{c_k^j}\leq C\ell_j^{3/2}$. 
To this end we note that since  $b_0^j(P_g^{\ker\partial \ell_j} (\Om^j))=0$ we may apply 
 \eqref{eq:orth-Four1} as well as the estimates \eqref{est:RT-neg1-new} on $\Om^j$ from Proposition \ref{prop:RT-new} to bound 
\beqas
   \Vert P_g^{\ker\partial \ell_j} (\Om^j)\Vert^2_{L^2(M,g)}&= \langle P_g^{\ker\partial \ell_j} (\Om^j), \Om^j\rangle_{L^2(M,g)}
 \\
 &=  \langle  P_g^{\ker\partial \ell_j} (\Om^j), \Om^j - b_0^j(\Om^j) dz^2 \rangle_{L^2(\Col(\si^j))}
   + \langle P_g^{\ker\partial \ell_j} (\Om^j), \Om^j\rangle_{L^2(M\setminus \Col(\si^j))}\\
   &\leq C\ell_j^{3/2} \Vert P_g^{\ker\partial \ell_j} (\Om^j)\Vert_{L^1(M,g)} 
   \eeqas
and hence to conclude that indeed
\beq \label{est:proj-ker}
   \Vert P_g^{\ker\partial \ell_j} (\Om^j)\Vert_{L^2(M,g)}\leq C\ell_j^{3/2}.
\eeq
To improve the obtained bound of $\abs{c_k^j}\leq C\ell_j^{3/2}$
we now consider the inner product of \eqref{eq:writ_second_part_Psi} with $\Om^k\in \text{ker}(\partial \ell_j)$ which, thanks to the estimates on $\Om^j$ from  Proposition~\ref{prop:RT-new}, yields
\begin{align*}
|c_k^j|\leq & |\langle \Om^k, a_j\i  P_g^{\ker\partial \ell_j} (\Om^j) \rangle| +\sum_{i\neq j,k} |c_i^j|\cdot |\langle \Om^k,\Om^i \rangle| \leq C \abs{\langle \Om^k,\Om^j\rangle} 
 +C \sum_{i\neq j,k} \ell_j^{3/2}\ell_k^{3/2} \ell_i^{3/2}\\
\leq& C\ell_j^{3/2}\ell_k^{3/2}
\end{align*}
as claimed in \eqref{est:coeff-Psi}.
We also remark that the claims \eqref{eq:write-Psi-with-Om} and \eqref{est:coeff-Psi} of the theorem, which we have just now established, imply in particular that \eqref{est:lemma-Psi1} holds true since $\norm{\Om^k}_{L^\infty(M,g)}\leq C\ell_k^{-1/2}$, compare \eqref{est:Linfty_Om_thick}.

It remains to show the bound \eqref{est:lemma-Psi4} on
 $\norm{\Psi^j}_{L^2}$ which, by \eqref{eq:inner-prod-Psi}, is given as
$$ \norm{\Psi^j}_{L^2(M,g)}={8\pi\i}b_0^j(\tfrac{\Psi^j}{\norm{\Psi^j}_{L^2(M,g)}})=a_j{8\pi}b_0^j(\Om^j).
$$ 
Combining the bound 
\eqref{est:RT-neg3-new} on $b_0^j(\Om^j)$ with the estimate \eqref{est:coeff-Psi} on $a_j$ yields
\beqas 
 \left|  \Vert \Psi^j\Vert_{L^2(M,g)} - 8\pi \Vert dz^2\Vert_{L^2(\Col(\si^j))}^{-1}\right|&\leq
 8\pi a_j \abs{b_0^j(\Om^j) -\Vert dz^2\Vert_{L^2(\Col(\si^j))}^{-1}}
 +
 8 \pi\abs{1-a_j}\cdot \Vert dz^2\Vert_{L^2(\Col(\si^j))}^{-1} \\
 &
  \leq C\ell_j^3 \norm{dz^2}_{L^2(\Col(\si^j))}^{-1}\leq C\ell_j^{9/2}
\eeqas as claimed,
where we used \eqref{est:sizes_on_collars} respectively \eqref{est:dz-lower} in the last step.
 \end{proof}

\subsection{Proof of Theorem \ref{thm:La} on the elements dual to $d\ell$}
\label{subsec:holo_la}
$ $

In this section we prove the desired properties of the elements $\La^j$ in two steps: In a first step we will construct a tensor $h^j\in T_g\M$ which induces the desired change of the Fenchel-Nielsen coordinates and hence 
projects down onto 
$\Rea \La^j$ under the projection $P_g^H\colon T_g\M\to H(g)=\Rea(\Hol(M,g))$ 
but is itself not horizontal. In a second step we will then analyse its projection using the estimates on the dual basis $\Th^j$ to the complex differentials $\partial \ell_j$ proven in Lemma~\ref{lemma:Om}.

Given a simple closed geodesic $\si^j\in \mathcal{E}$ we construct such an element $h^j\in T_g\M$ as follows. 
We decompose $M$ into pairs of pants by cutting along the curves in $\mathcal{E}$ and consider the (closures of the) pair(s) of pants  $P_i$ for which $\si^j$ is a boundary curve, where either $i=1,2$ with $\si^j$ corresponding to one boundary curve of each of the $P_i$'s, or $i=1$ with $\si^j$ corresponding to two of the boundary curves of $P_1$.

\begin{minipage}{0.44\textwidth}
We decompose these one or two pairs of pants further by cutting along the seams, i.e. the shortest geodesics between the boundary curves, resulting in one respectively two pairs of identical geodesic rectangular hexagons.
\newline

On these hexagons we can consider an evolution of the metric as described in the following Lemma~\ref{lemma:hex} where one should think of the sides $\Gamma_{a,b,c}$ as the seams of such a pair of pants $P_i$, while $\gamma_{a,b,c}$ correspond to (half)curves from $\mathcal{E}$ which have constant length unless they correspond to half of $\si^j$, compare Figure~\ref{fig_hexa_0}.
\end{minipage}\hfill
\begin{minipage}{0.64\textwidth}\begin{center}
\begin{tikzpicture}[scale=0.52]
 \draw (0,5)  .. controls (1,5.6) .. (2,7)  .. controls (3,6.52) ..  (4.8,6.87)  .. controls (5.5, 5.2) .. (6,4.5)  .. controls (5,3.6)  .. (4,1)   .. controls (3, 1.5) .. (2,1) .. controls (1.4,3) ..(0,5);

\begin{scope} 
\clip (0,5) node[left] {\large $p_1$}  .. controls (1,5.6) .. (2,7) node[above] {\large $p_2$} .. controls (3,6.52) ..  (4.8,6.87)  node[above] {\large  $p_3$} .. controls (5.5, 5.2) .. (6,4.5) node[right] {\large $p_4$} .. controls (5,3.6)  .. (4,1) node[below] {\large $p_5$}  .. controls (3, 1.5) .. (2,1) node[left] {\large $p_6$}.. controls (1.4,3) ..(0,5);

\draw (0,5) circle (0.2cm);
\draw (0.1,4.97) circle (0.01cm);

\draw (2,7) circle (0.2cm);
\draw (2.03,6.88) circle (0.01cm);

\draw (4.8,6.87) circle (0.2cm);
\draw (4.77,6.75) circle (0.01cm);

\draw (6,4.5) circle (0.2cm);
\draw (5.9,4.5) circle (0.01cm);

\draw  (4,1) circle (0.2cm);
\draw (3.95,1.1) circle (0.01cm);

\draw (2,1) circle (0.2cm);
\draw (2.05,1.1) circle (0.01cm);
\end{scope}

\draw (0.6,4.15) .. controls (2.3, 5) .. (3,6.64);
\draw (4,6.69) .. controls (4.4, 4.7) .. (5.3,3.8) ;
\draw (1.1,3.4) .. controls (3, 4) .. (5,3.4) ;

 \begin{scope} 
 \clip (0.6,4.15) node[left] {\large $q_1$}.. controls (2.3, 5) .. (3,6.64) node[above] {\large $q_2$}  .. controls (3.6, 6.6) ..  (4,6.75) node[above] {\large $q_3$}.. controls (4.4, 4.7) .. (5.3,3.8) -- (5,3.4) .. controls (3, 4) .. (1.1,3.4) -- (0.6, 4.15);

\draw (0.6,4.15) circle (0.2cm);
\draw (3,6.64) circle (0.2cm);
\draw (4,6.69) circle (0.2cm);
\draw (5.3,3.8) circle (0.2cm);
\draw (5,3.4) circle (0.2cm);
\draw (1.1,3.4) circle (0.2cm);

\end{scope}

\draw (0.7,4.1) circle (0.01cm);
\draw (3.1,6.55) circle (0.01cm);
\draw (3.9,6.6) circle (0.01cm);
\draw (5.2,3.78) circle (0.01cm);
\draw (4.95,3.5) circle (0.01cm);
\draw (1.15,3.5) circle (0.01cm);

\draw (0.8,6) node {\large $\gamma_a$};
\draw (5.6,5.7) node {\large $\gamma_b$};
\draw (3,1) node {\large $\gamma_c$};

\draw (1.6,5.5) node {\large $R_a$};
\draw (4.8,5.3) node {\large $R_b$};
\draw (3,2.7) node {\large $R_c$};
\draw (3,4.7) node {\large $\tilde{H}_0$};

\draw (5.5,3.1) node {\large $\Gamma_a$};
\draw (0.7,3.3) node {\large $\Gamma_b$};
\draw (3.3,7) node {\large $\Gamma_c$};

\end{tikzpicture}\\[-0.2cm]

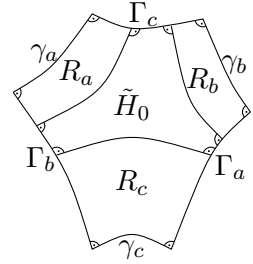
\captionof{figure}{A rectangular hexagon corresponding to half of a pair of pants $P_i$. Here $\Gamma_.$ are the seams of $P_i$ and $\gamma_.$ are (half)curves from $\mathcal{E}$. The collar parts $R_.$ and the interior hexagon $\tilde{H}_0$ are defined in Lemma~\ref{lemma:hex}.
}\label{fig_hexa_0}
                                \end{center}
\end{minipage}

\begin{lemma}\label{lemma:hex}
Let $\bar L>0$ be any given number and let $(H_t)_t$ be a family of rectangular, geodesic hexagons in the hyperbolic plane $(\Hyp, g_\Hyp)$
whose sides $\gamma_{a,t}, \Gamma_{c,t}, \gamma_{b,t},\Gamma_{a,t}, \gamma_{c,t}, \Gamma_{b,t}$ satisfy either 
\beq \label{ass:hex-1}
L_{g_\Hyp}(\gamma_{a,t})=a_t\equiv a, \qquad L_{g_\Hyp}(\gamma_{b,t})=b_t\equiv b, \qquad L_{g_\Hyp}(\gamma_{c,t})=c_t=\frac12(\ell+t)
\eeq
or 
\beq \label{ass:hex-2}
L_{g_\Hyp}(\gamma_{a,t})=a_t\equiv a, \qquad L_{g_\Hyp}(\gamma_{b,t})=b_t=L_{g_\Hyp}(\gamma_{c,t})=c_t= \frac12(\ell+t)
\eeq
for numbers $0<a_t,b_t,c_t\leq \frac12 \bar L$. 
We set $c_5\define \frac{w_{\bar{L}}}{4}$ for $w_{\bar{L}}$ the width of the collar around a simple closed geodesic of length $\bar L$ characterised by \eqref{eq:width}, and consider the 
subset $R_a(t)$  of the collar around $\gamma_{a,t}$ given by 
$$
R_a(t)=\{ p\in \Col(\ga_{a,t}): \dist(p, \partial \Col(\ga_{a,t}))\geq c_5\in (0,\tfrac{w_{2a_t}}{2})\}$$
as well as the analogue subsets $R_b(t)$ and $R_c(t)$  of $\Col(\gamma_{b,t})$ and $\Col(\gamma_{c,t})$, see Figure~\ref{fig_hexa_0}.
\newline
Then there exists 
a family of diffeomorphisms $F_t\colon (H_0, g_\Hyp)\to (H_t,g_\Hyp)$, which is generated by a smooth vector field $X$ on $H_0$, so that the following holds true: 
\begin{enumerate}[(i)]
\item\label{hex_iso} $F_t$ is an isometry from $R_a\define R_a(0)$ to $R_a(t)$, and in the setting of  \eqref{ass:hex-1} also from $R_b\define R_b(0)$ to $R_b(t)$.
\item\label{hex_collar} $F_t$ maps $R_c\define R_c(0)$ onto $R_c(t)$ and takes the form $$F_t(s,\theta)=(f_{c,t}(s),\theta) \qquad \text{ on } R_c$$
with respect to the corresponding collar coordinates, where  $f_{c,t}(\cdot)\colon (-X(2c),X(2c))\to (-X(2c_t),X(2c_t))$ is an odd function.  
 In the setting of  \eqref{ass:hex-2} the same property holds also for $R_b$.
\item\label{hex_inner} The change $\pt g(0)=L_Xg$ 
of the induced metrics $g(t)=F_t^*g_{\mathbb{H}}$ on 
$$\tilde H_0\define  H_0\setminus (R_a\cup R_b\cup R_c)$$ is bounded by 
$$\norm{\pt g(0)}_{L^\infty(\tilde H_0)}\leq C \ell$$
for a constant $C$ that depends only on  $\bar L$.
\item\label{hex_sym} On $\Gamma_{a,b,c}$ we have that the normal derivatives of 
odd order $\big(\frac{\partial}{\partial n_{\Gamma_{\cdot}}}\big)^{2j+1} X$, $j=0,1,\ldots$, vanish identically.
\end{enumerate}
\end{lemma}

Returning to the construction of a  tensor $h^j\in T_g\M$ that induces the desired change of the Fenchel-Nielsen coordinates
we note that \eqref{ass:hex-1} corresponds to having two pairs of pants adjacent to $\si^j$,  while the case that $\si^j$ only has one adjacent pairs of pants for which it corresponds to two boundary curves is treated by considering the case \eqref{ass:hex-2}.

We furthermore remark that \eqref{hex_sym} imposes compatibility conditions  on the sides $\Gamma_{a,b,c}$ of the hexagon corresponding to seams of the pairs of pants that guarantee that the resulting tensor $\pt g(0)=L_Xg$ can be extended to a smooth tensor on $P_{1,2}$ respectively to $P_1$  by symmetry. Since the function $f_{\cdot}$ in part \eqref{hex_collar} of Lemma~\ref{lemma:hex} is odd we may glue the resulting tensors on $P_{1,2}$ respectively on $P_1$
along $\si^j$ to obtain a smooth tensor on 
the closed set $P\subset M$ which corresponds to the (union of the) pairs of pants that are adjacent to $\si^j$.
By \eqref{hex_iso} $\pt g(0)$ is zero near the boundary of $P$ so we may extend the obtained tensor by zero to the rest of $M$ to finally obtain  an element $h^j$ of $T_g\M$ which, thanks to \eqref{ass:hex-1} resp. \eqref{ass:hex-2}, induces the desired change of the Fenchel-Nielsen coordinates.
As a result we therefore obtain  

\begin{cor}\label{cor:h}
 Let $(M,g)$ be a hyperbolic surface, let $(\ell_i, \psi_i)_{i=1}^{3(\gamma-1)}$ be the Fenchel-Nielsen coordinates corresponding to a collection $\mathcal{E}$ of disjoint simple closed geodesics 
 which decomposes $M$ into pairs of pants and let $\eta\in (0,\arsinh(1))$, $\bar L$ be constants for which \eqref{ass:eta} and \eqref{ass:upperbound} hold true. 
  Then for every $j\in \{1,\ldots, 3(\gamma-1)\}$ there exists a tensor $h^j\in T_g\mathcal{M}_{-1}$ 
 such that 
 \beq \label{evol_l}
 d\ell_i(h^j)=\de_{i}^j \text{ and } d\psi_i(h^j)=0 \text{ for every } i=1,\ldots, 3(\gamma-1)
 \eeq
so that on the subset 
$$\Col_{c_5}(\si^j)\define  \{p\in \Col(\si^j): \text{dist}(p,\partial \Col(\si^j))\geq c_5\},$$ $c_5=c_5(\bar L)>0$ as in Lemma~\ref{lemma:hex}, of the collar $\Col(\si^j)$
the tensor $h^j$ takes the form 
\beq \label{eq:form-h-collar}
h^j=\xi_1(s)(ds^2-d\th^2) +\xi_2(s)(ds^2+d\th^2) \eeq 
with respect to collar coordinates $(s,\th)$
while 
\beqs \supp(h^j)\setminus \Col_{c_5}(\si^j)\subset \tilde \eta \thick (M,g)
\eeqs
for a number $\tilde\eta=\tilde \eta(\bar L, \eta)>0$
and so that for  a constant $C=C(\eta,\bar L)$
\beq 
\label{est:h_L1}
\Vert h^j\Vert_{L^\infty(M\setminus \Col_{c_5}(\si^j))}\leq C\ell_j.\eeq
\end{cor}

Before giving the proof of Lemma~\ref{lemma:hex} we complete the

\begin{proof}[Proof of Corollary~\ref{cor:h}]
It remains to show that the tensor $h^j\in T_g\M$, which we obtained above by gluing together the tensors $\pt g=L_Xg$ from Lemma~\ref{lemma:hex}, has the desired properties. 

We have already observed that $h^j$ induces the desired change  \eqref{evol_l} of the Fenchel-Nielsen coordinates and note that 
\eqref{est:h_L1} is an immediate consequence 
 of part \eqref{hex_inner} of  Lemma~\ref{lemma:hex}. 
 Furthermore $\pt g=L_Xg$ has the desired form \eqref{eq:form-h-collar} on  $\Col_{c_5}(\si^j)$ as on this set the vector field $X$ generating $F_t$ has the form $X(s,\theta)=\xi(s) \frac{\partial}{\partial s}$ for some function $\xi$, while the metric is of course given by $g=\rho^2(s)(ds^2+d\th^2)$.
We furthermore
note that part \eqref{hex_iso} of of  Lemma~\ref{lemma:hex} ensures that
points in $\supp(h^j)\cap \Col(\si^i)$, $i\neq j$, have distance no more than $c_5$ from $\partial \Col(\si^i)$ 
which, by \eqref{est:rho-ends-of-collar}, means that their injectivity radius is bounded from below by a constant depending only on $c_5$ and hence $\bar L$. Since $M\setminus \bigcup_i \Col(\si^i)\subset \eta\thick(M,g)$, compare \eqref{ass:eta}, this finally yields the claim on the support of $h^j$. 
\end{proof}

\begin{proof}[Proof of Lemma~\ref{lemma:hex}]
Given a family $H_t$ of hexagons as described in the lemma we denote by 
 $p_i(t)$ the vertices of the hexagon $H_t$ as shown in 
Figure~\ref{fig_hexa} and note that 
since the length of the geodesics $\ga_{a,t}$ is constant, we may assume without loss of generality that the corresponding sides of $H_0$ and $H_t$ coincide, i.e. that $p_1(t)=p_1$ and $p_2(t)=p_2$ for every $t$ where we abbreviate $p_i=p_i(0)$.
Together with the prescribed lengths $a_t, b_t, c_t$ of the alternate sides $\ga_{a,t}, \ga_{b,t}$ and $\ga_{c,t}$ this determines the subset $H_t$ in the hyperbolic plane.

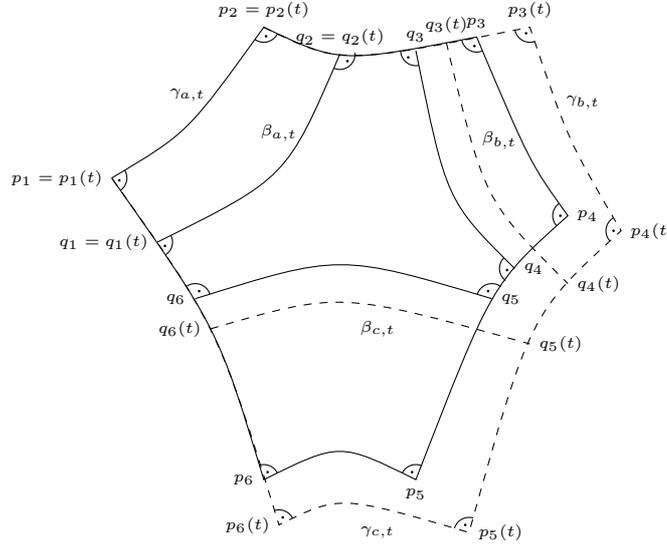
\begin{figure}[h!]
\centering  \begin{tikzpicture}
 \draw (0,5) node[left] {\tiny $p_1=p_1(t)$}  .. controls (1,5.6) .. (2,7) node[above] {\tiny $p_2=p_2(t)$} .. controls (3,6.52) ..  (4.8,6.87)  node[above] {\tiny  $p_3$} .. controls (5.5, 5.2) .. (6,4.5) node[right] {\tiny $p_4$} .. controls (5,3.6)  .. (4,1) node[below] {\tiny $p_5$}  .. controls (3, 1.5) .. (2,1) node[left] {\tiny $p_6$}.. controls (1.4,3) ..(0,5);

\begin{scope} 
\clip (0,5) node[left] {\tiny $p_1=p_1(t)$}  .. controls (1,5.6) .. (2,7) node[above] {\tiny $p_2=p_2(t)$} .. controls (3,6.52) ..  (4.8,6.87)  node[above] {\tiny  $p_3$} .. controls (5.5, 5.2) .. (6,4.5) node[right] {\tiny $p_4$} .. controls (5,3.6)  .. (4,1) node[below] {\tiny $p_5$}  .. controls (3, 1.5) .. (2,1) node[left] {\tiny $p_6$}.. controls (1.4,3) ..(0,5);

\draw (0,5) circle (0.2cm);
\draw (0.1,4.97) circle (0.01cm);

\draw (2,7) circle (0.2cm);
\draw (2.03,6.88) circle (0.01cm);

\draw (4.8,6.87) circle (0.2cm);
\draw (4.77,6.75) circle (0.01cm);

\draw (6,4.5) circle (0.2cm);
\draw (5.9,4.5) circle (0.01cm);

\draw  (4,1) circle (0.2cm);
\draw (3.95,1.1) circle (0.01cm);

\draw (2,1) circle (0.2cm);
\draw (2.05,1.1) circle (0.01cm);
\end{scope}

\draw[dashed] (2,7)  .. controls (3,6.5) ..  (5.5,7) node[above] {\tiny $p_3(t)$}
 .. controls (6, 5.5) .. (6.7,4.3) node[right] {\tiny $p_4(t)$}
 .. controls (5.5,3.2)  .. (4.7,0.3) node[right] {\tiny $p_5(t)$}  .. controls (3, 0.8) .. (2.2,0.4) node[left] {\tiny $p_6(t)$}.. controls (1.4,3.02) ..(0,5);

\begin{scope} 
 \clip (2,7)  .. controls (3,6.5) ..  (5.5,7) node[above] {\tiny $p_3(t)$}
 .. controls (6, 5.5) .. (6.7,4.3) node[right] {\tiny $p_4(t)$}
 .. controls (5.5,3.2)  .. (4.7,0.3) node[right] {\tiny $p_5(t)$}  .. controls (3, 0.8) .. (2.2,0.4) node[left] {\tiny $p_6(t)$}.. controls (1.4,3.02) ..(0,5);
\draw (5.5,7) circle (0.2cm);
\draw (5.45,6.92) circle (0.01cm);

\draw (6.7,4.3) circle (0.2cm);
\draw (6.6,4.3) circle (0.01cm);

\draw (4.7,0.3) circle (0.2cm);
\draw (4.65,0.4) circle (0.01cm);

\draw (2.2,0.4) circle (0.2cm);
\draw (2.25,0.5) circle (0.01cm);
\end{scope}

\draw (0.6,4.15) node[left] {\tiny $q_1=q_1(t)$}.. controls (2.3, 5) .. (3,6.64) node[above] {\tiny $q_2=q_2(t)$};

\draw (4,6.69) node[above] {\tiny $q_3$}.. controls (4.4, 4.7) .. (5.3,3.8) node[right] {\tiny $q_4$};

\draw[dashed] (4.4,6.8) node[above] {\tiny $q_3(t)$}.. controls (4.9, 4.7) .. (6,3.6) node[right] {\tiny $q_4(t)$};

\draw (1.1,3.4) node[left] {\tiny $q_6$}.. controls (3, 4) .. (5,3.4) node[right] {\tiny $q_5$};

\draw[dashed] (1.3,3) node[left] {\tiny $q_6(t)$}.. controls (3, 3.5) .. (5.5,2.8) node[right] {\tiny $q_5(t)$};

 \begin{scope} 
 \clip (0.6,4.15) node[left] {\tiny $q_1=q_1(t)$}.. controls (2.3, 5) .. (3,6.64) node[above] {\tiny $q_2=q_2(t)$}  .. controls (3.6, 6.7) ..  (4,6.69) node[above] {\tiny $q_3$}.. controls (4.4, 4.7) .. (5.3,3.8) -- (5,3.4) .. controls (3, 4) .. (1.1,3.4) -- (0.6, 4.15);

\draw (0.6,4.15) circle (0.2cm);
\draw (3,6.64) circle (0.2cm);
\draw (4,6.69) circle (0.2cm);
\draw (5.3,3.8) circle (0.2cm);
\draw (5,3.4) circle (0.2cm);
\draw (1.1,3.4) circle (0.2cm);

\end{scope}

\draw (0.7,4.1) circle (0.01cm);
\draw (3.1,6.55) circle (0.01cm);
\draw (3.9,6.6) circle (0.01cm);
\draw (5.2,3.78) circle (0.01cm);
\draw (4.95,3.5) circle (0.01cm);
\draw (1.15,3.5) circle (0.01cm);

\draw (1,6.1) node {\tiny $\gamma_{a,t}$};
\draw (6.2,6) node {\tiny $\gamma_{b,t}$};
\draw (3.5,0.3) node {\tiny $\gamma_{c,t}$};

\draw (2.2,5.6) node {\tiny $\beta_{a,t}$};
\draw (5.1,5.5) node {\tiny $\beta_{b,t}$};
\draw (3.5,3) node {\tiny $\beta_{c,t}$};

\end{tikzpicture}
\caption{Hexagons $H_0$ and $H_t$:
All lines are geodesics except $\beta_{a,b,c}$ which are curves of constant geodesic curvature.
}\label{fig_hexa}
\end{figure}

We set $c_5=\frac{w_{\bar{L}}}{4}$ and note that 
$c_5\leq \frac{w_{2i}}{4}< 
\frac{w_{2i}}{2}=\dist(\gamma_{i}, \partial \Col(\gamma_{i}))$ for $i\in \{a,b,c\}$ as we have assumed that $a_t,b_t,c_t\leq \frac{\bar L}2$.
We may thus consider the curve $\beta_{a,t}$ of all 
points in the collar $\Col(\gamma_{a,t})$ whose distance to $\partial \Col(\gamma_{a,t})$ is $c_5$, i.e. whose distance to $\gamma_{a,t}$ is equal to $\frac{w_{2a_t}}{2}-c_5$.
These curves $\beta_{a,t}$  meet the geodesic boundary curves $\Gamma_{b,t}$ and $\Gamma_{c,t}$ orthogonally in two points which we denote by $q_1(t)$ and $q_2(t)$.
The set $R_a$ introduced in the lemma then corresponds to the rectangular quadrangle with vertices $p_1,p_2, q_2,q_1$ whose boundary curves are geodesics except for $\beta_{a}$ along which the geodesic curvature is constant, but non-zero. 
We introduce the analogue notation also on the collars around $\gamma_{b,t}$ and $\gamma_{c,t}$, compare Figure~\ref{fig_hexa}, and denote by $\tilde H_t$ the 
inner rectangular hexagon.

We furthermore remark that the collar region $\Col(\gamma_{a,t})$ around $\gamma_{a,t}$ is isometric to 
$\big([0,X(2a_t))\times [0,\pi], \rho_{2a_t}(s)^2(ds^2+d\th^2)\big)$, $X,\rho$ as always given by \eqref{eq:rho-X}, 
as it corresponds to a quarter of a collar around a simple closed geodesic of length $2a_t$ in a closed hyperbolic surface.
In collar coordinates, the curve $\beta_{a,t}$ corresponds to the semicircle $\{Z_{c_5}(2a_t)\}\times [0,\pi]$, where $Z_{c}(\ell)= \frac{2\pi}{\ell}(\frac{\pi}{2}-Y_c(\ell))$ is characterised by 
$$c=\int_{Z_c(\ell)}^{X(\ell)} \rho(s) ds =\int^{\frac\pi2-\frac{\ell}{2\pi} Z_c(\ell)}_{\frac\pi2-\frac\ell{2\pi} X(\ell)}\tfrac{1}{\sin(t)}dt = \log\left(\tan(\thalf Y_c(\ell))\right)-\log\left(\tan(\thalf\arctan(\sinh(\tfrac{\ell}{2})))\right)
$$
and hence 
\beq\label{eq:Yc}
Y_c(\ell)= 2 \arctan \big( e^c \tan (\thalf Y_0(\ell))\big) \quad \text{where\ } Y_0(\ell)=\arctan \sinh (\tfrac{\ell}{2}).
\eeq
We now construct the desired diffeomorphism $F_t\colon H_0\to H_t$ as follows: 

On the rectangular subsets $R_{\cdot}$ which contain a boundary curve $\gamma_{\cdot}$ whose length is fixed, 
i.e. on $R_a$ and $R_b$ in the setting of \eqref{ass:hex-1} respectively only on $R_a$ in the setting of \eqref{ass:hex-2}, 
the diffeomorphism $F_t\colon R_a\to R_a(t)$ is uniquely determined by the condition that it is an isometry. 

The other rectangular sets $R_{\cdot}$
correspond to subsets of collars 
around geodesics of length  $\frac{1}2(\ell+t).$
Here we can choose $F_t\colon R_{b,c}\to R_{b,c}(t)$ explicitly e.g. as a linear map
$F_t(s,\theta)=(\tfrac{Z_{c_5}(\ell+t)}{Z_{c_5}(\ell)}\cdot s, \theta)$
in the corresponding collar coordinates, in which $R_.(t)$ is given by $[0,Z_{c_5}(\ell+t)]\times [0,\pi]$.

It hence remains to show that we can define $F_t$ on the inner hexagon $\tilde H_0$ in such a way that $F_t$ is smooth, so that the induced change of the metric 
$\frac{d}{dt}|_{t=0} F_t^*g_\Hyp=L_X g_{\Hyp}$ is of order $O(\ell)$ as required in \eqref{hex_inner} and so that $X$ satisfies the symmetry conditions from \eqref{hex_sym}.
The main step to prove this is to show that 
 the change of all geometric quantities characterising $\tilde H_t$, i.e. the distances $\text{dist}(q_i(t),q_{i+1}(t))$, $i=1,\ldots, 6$ (where we set $q_7=q_1$) as well as the geodesic curvatures of those boundary curves $\beta_{a,b,c}$ of $\tilde H_t$ which are not geodesics, is of order $O(\ell)$:
 
  This trivially holds true for the curves $\beta_{a,t}$, 
 and  in case of \eqref{ass:hex-1} also for $\beta_{b,t}$,
 since $a_t$ is constant and hence these curves are isometric to one-another.
  
For $\beta_{c,t}$ (and so in case of \eqref{ass:hex-2} by symmetry also for $\beta_{b,t}$) this can be seen as follows:
We write 
$\beta_{c,t}$ in collar coordinates as  $\{Z_{c_5}(\ell)\}\times [0,\pi]$,  $Z_{c_5}(\ell)=\tfrac{2\pi}{\ell}(\tfrac{\pi}{2}-Y_{c_5}(\ell))$ where $Y_{c_5}$ is given by \eqref{eq:Yc} and hence satisfies in particular 
\beq \label{eq:Yc-main}
Y_{c_5}(\ell)=e^{c_5}\tfrac{\ell}{2}+O(\ell^3) \text{ and } Y_{c_5}'(\ell)=e^{c_5}\tfrac{1}{2}+O(\ell^2).\eeq
As the geodesic curvature 
of a curve $\{s\}\times S^1$ in a collar $\Col(\si)$ around a geodesic of length $\ell$ is
 given by
\beqs\label{eq:geod-curv}
\kappa_g= \big\langle \na_{\rho^{-1}\tfrac{\partial}{\partial \theta}}(\rho^{-1}\tfrac{\partial}{\partial \theta}) , \rho^{-1}\tfrac{\partial}{\partial s}\big\rangle = \rho^{-1}\Gamma_{\th\th}^s=-\rho^{-2}(s)\rho'(s)=-\sin(\tfrac{\ell}{2\pi}s),
\eeqs
we thus obtain that the curvature 
$\kappa_{c,t}=-\sin(\tfrac{\ell+t}{2\pi} Z_{c_5}(\ell+t))=-\cos(Y_{c_5}(\ell+t))$
of
 $\beta_{c,t}$  
satisfies
$$\abs{\tfrac{d}{dt} \kappa_{{c,t}}}
=\sin(Y_{c_5}(\ell))\cdot \abs{\tfrac{d Y_{c_5}(\ell)}{d\ell}}\leq C
\ell,$$
where here and in the following all time derivatives are evaluated in $t=0$.
As the length of $\beta_{c,t}$ is $L(\beta_{c,t})=\pi\rho(Z_{c_5}(\ell+t))=\frac{\ell +t}{2\sin(Y_{c_5}(\ell +t))}$ also its change is 
only of order 
$$\babs{\tfrac{d}{dt} L(\beta_{c,t})}
=\babs{\frac{\sin( Y_{c_5}(\ell))-\ell Y_{c_5}'(\ell) \cos (Y_{c_5}(\ell +t))}{2\sin^2( Y_{c_5}(\ell))}}\leq C\ell,
$$
compare \eqref{eq:Yc-main}. 

The change of the lengths of the sides of $\tilde H_t$ which are subsets of the boundary curves $\Gamma_{a,b,c}$ of the original hexagon, and hence geodesics, can be computed using standard formulas from hyperbolic trigonometry as found e.g. in  \cite[p.454]{Buser} as we explain in the following:

We first consider the case \eqref{ass:hex-1} in which both $a_t$ and $b_t$ are constant. 
Here we may use 
$$ \cosh (L(\Gamma_{c,t}))= \frac{\cosh c_t + \cosh a\cosh b}{\sinh a\sinh b}$$
to express the length of the side opposite $\gamma_{c,t}$ as 
\beqas 
L(\Gamma_{c,t}) &= \log \left(\cosh c_t + \cosh a\cosh b + \sqrt{ (\cosh c_t + \cosh a\cosh b)^2-\sinh^2 a\sinh^2 b}\right)\\&\quad - \log (\sinh a\sinh b). \eeqas
We we note that the last term is constant, while the term in the square root is bounded away from zero by $\cosh^2 a\cosh^2 b-\sinh^2 a\sinh^2 b\geq 1$ so 
\begin{align}\label{eq:Gamma_deriv}\abs{\tfrac{d}{dt} L(\Gamma_{c,t}) }\leq C \sinh(c_t)\cdot \abs{\tfrac{d}{dt} c_t}\leq 
C\ell.\end{align}
Here and in the following $C$ denotes a generic constant that depends at most on the upper bound $\frac12\bar L$ on $a,b,c$.

As $\dist(q_i(t),p_i(t))=\frac{w}{2}-c_5$, $w$ the width of the corresponding collar, we have 
$$\dist(q_2(t),q_3(t)) = L(\Gamma_{c,t})+2c_5-\tfrac{w_a}{2}-\tfrac{w_b}{2}$$
and hence, still considering only the setting of \eqref{ass:hex-1} where $a,b$ and thus $w_a,w_b$ are constant, 
$$\abs{\tfrac{d}{dt} \dist(q_2(t),q_3(t))}= \abs{\tfrac{d}{dt} L(\Gamma_{c,t}) }\leq C\ell.$$
We now observe that the terms in
$\dist(q_4(t),q_5(t)) = L(\Gamma_{a,t})+2c_5-\tfrac{w_a}{2}-\tfrac{w_{c_t}}{2}$
are given by 
\beqas
L(\Gamma_{a,t}) &=\arsinh\bigg(\frac{\sinh L(\Gamma_{c,t}) \sinh a}{\sinh c_t}\bigg)
\\
&=\log\big[ \sinh L(\Gamma_{c,t})\sinh a + \sqrt{\sinh^2 L(\Gamma_{c,t}) \sinh^2 a + \sinh^2 c_t} \big] -\log\sinh c_t
\eeqas
respectively, using \eqref{eq:width}, by  
$$\tfrac{w_{c_t}}{2}=\arsinh(\sinh^{-1}c_t)=
\log\big[ 1+\sqrt{1+ \sinh^2 c_t}\big]-\log \sinh c_t.$$
The terms $\log\sinh c_t$, whose derivative would be of order $1$ rather than $\ell$, thus cancel and
\beqas 
\dist(q_4(t),q_5(t)) 
=&\ 
 -\log\big[ 1+\sqrt{1+ \sinh^2 c_t}\big] +2c_5-\frac{w_a}{2}
 +\log(\sinh L(\Gamma_{c,t}))
 +\log(\sinh a)
\\
& +\log\bigg[ 1+ \sqrt{1 +\sinh^2 c_t\sinh^{-2}(a)\sinh^{-2}( L(\Gamma_{c,t})) }\bigg],
\eeqas
where we note that $L(\Gamma_{c,t})\geq \frac{w_a}{2}$ and hence 
$\sinh(a)\sinh(L(\Gamma_{c,t}))\geq \sinh(a)\cdot \sinh(\frac{w_a}{2})=
1$.

Since $L(\Gamma_{c,t})\geq w_{\bar L}$ while  $\ddt L(\Gamma_{c,t})$  is controlled by \eqref{eq:Gamma_deriv},  we may thus estimate
$$\abs{\tfrac{d}{dt} \dist(q_4(t),q_5(t))}\leq C \cdot \sinh(c_t)+C\cdot \abs{\tfrac{d}{dt} L(\Gamma_{c,t})} \leq C\ell.$$
The same argument also applies for 
$\abs{\tfrac{d}{dt} \dist(q_6(t),q_1(t))}$.

This completes the proof that the change of all geometric quantities of $\tilde H_t$ are of order $O(\ell)$ in the case \eqref{ass:hex-1}. 
In the setting of \eqref{ass:hex-2} we may argue similarly, now additionally using the symmetry of the hexagon $H_t$, to obtain the analogue estimates. 

Returning to the construction of the diffeomorphism $F_t$ on $\tilde H_0$ we note that 
the derived bounds imply in particular that 
the 
points $q_3(t)$ and $q_6(t)$ in the hyperbolic plane move only by 
$\abs{\tfrac{d}{dt} q_3(t)}_{g_\Hyp}+\abs{\tfrac{d}{dt} q_6(t)}_{g_\Hyp}\leq C\ell$ 
along the geodesics $p_2p_3$ respectively $p_1p_6$ of the original hexagon. 
We will later choose the vector field $X$ along these sides of $\tilde H_0$ as
a suitable interpolation between $X(q_1)=0$ and $X(q_6)=O(\ell)$,
respectively $X(q_2)$ and $X(q_3)$, but before that consider the curves $\beta_{c,b}$, 
which we
parametrise by the corresponding collar coordinate $\theta$, i.e. 
proportionally to arc length $\abs{\beta'_{c,t} (\theta)}=\pi^{-1} L(\beta_{c,t})$ over the interval $[0,\pi]$. 
On these curves 
$F_t$ is already determined by the choice of $F_t$ on the rectangles $R_{b,c}$ which requires that 
$F_t(\beta_{c}(\theta))=\beta_{c,t}(\theta)$ etc. Hence $X\circ \beta_{c}= \tfrac{d}{dt}\vert_{t=0} \beta_{c,t}$, respectively $X\circ \beta_{b} = \tfrac{d}{dt}\vert_{t=0} \beta_{b,t}$.

While we can control the change of the metric $L_Xg$ induced by the diffeomorphisms by working directly in collar coordinates, to obtain the necessary extension of $X$ to the interior hexagon $\tilde H_t$ we want to interpolate between the vector fields 
$X\circ \beta_{a}\equiv 0$, $X\circ \beta_{b}$ and $X\circ \beta_{c}$. So we need $C^1$-bounds on $X\circ \beta_{b}$ and $X\circ \beta_{c}$ which we obtain as follows.
As $\beta_{c,t}$ has constant geodesic curvature $\kappa_{c,t}$, it is characterised by 
its initial data $\beta_{c,t}(0)=q_6(t)$ and $\beta_{c,t}'(0)$,
which points in direction of the interior normal of $\partial H_t$ at  $q_6(t)$ and has length $\pi^{-1} L(\beta_{c,t})$, and the ODE
$$\nabla_{\beta_{c,t}'} \beta_{c,t}'=\kappa_{c,t}\cdot \abs{\beta'_{c,t}} \cdot (\beta_{c,t}')^\perp.$$ 
Here we denote by $v^\perp$ the vector obtained by rotating $v$ by $\pi/2$, by $\cdot'$ 
 derivatives with respect to $\theta$ and by  $\na$ covariant derivatives in $(\mathbb{H},g_{\mathbb{H}})$. 
We can write this equation equivalently as 
$$\beta_{c,t}''= \pi^{-1} L(\beta_{c,t}) \kappa_{c,t}\cdot  (\beta_{c,t}')^\perp-\Gamma(\beta_{c,t})(\beta'_{c,t}, \beta'_{c,t}).$$
where $\Gamma(p)(v,w)\define \Gamma_{ij}^k(p)v^iw^j \frac{\partial}{\partial x^k}$ is given in terms of the Christoffel-symbols of the hyperbolic plane. 
Differentiation in time yields that the vector field 
$Y(\th)\define X(\beta_{c,t}(\theta))= \frac{d}{dt}|_{t=0}\beta_{c,t}(\theta)$ satisfies a  second order ODE 
of the form 
$$Y''=f_{c,1}(\th)+f_{c,2}(\th,Y,Y')$$
where the first term is bounded by 
$\abs{f_{c,1}}\leq C\abs{\frac{d}{dt}|_{t=0}\kappa_{c,t}} +C \abs{\frac{d}{dt}|_{t=0}L(\beta_{c,t})}\leq C\ell$, 
 while the second term is linear in $Y$ and $Y'$ and has Lipschitz-constant bounded by $ C (1+L(\beta_{c})^{2}) \leq C(\bar L).$ 
Rewriting this as a system of first order linear ODEs and applying Gronwall's inequality hence allows us to bound 
$$\abs{Y(\theta)}+\abs{Y'(\theta)}\leq C \ell+ C\cdot (\abs{Y(0)}+\abs{Y'(0)})\leq C\ell$$
for a constant $C$ that depends only on the upper bound $\bar L$ on the sidelengths $a,b,c$.
The same argument applies of course also for $X\circ \beta_{b}$.

Having thus determined the size of $X$ on the boundary curves $\beta_{a,b,c}$ of $\tilde H_0$ we finally remark that the lengths of the other sides of $\tilde H_0$ are bounded away from zero by the constant $2c_5>0$
that depends only on $\bar L$ as the collars are disjoint.
This allows us to extend $X$ to a smooth vector field on all of $H_t$, chosen with the symmetries from \eqref{hex_sym},
for which the $C^1$-norm on $\tilde H_0$ is bounded by a fixed multiple of 
the $C^1$-norm on $\beta_{a}\cup \beta_{b}\cup \beta_c$, i.e. by $C\ell$ as claimed in \eqref{hex_inner}.
\end{proof}
We can now finally prove the properties of the holomorphic quadratic differentials $\La^j$ that are dual to the (real) differentials of the Fenchel-Nielsen length coordinates.

\begin{proof}[Proof of Theorem~\ref{thm:La}]
Let $h^j\in T_g\M$ be a tensor  that induces the desired change \eqref{def:Psi_La} 
of the Fenchel-Nielsen coordinates as obtained in Corollary~\ref{cor:h}. 
As its projection $P_g^H(h^j)$
 onto the horizontal space $H(g)$
differs from $h^j$
only by a Lie-derivative, it induces the same change \eqref{def:Psi_La} of the Fenchel-Nielsen coordinates. The desired element $\La^j$ of $\Hol(M,g)$ is hence characterised uniquely by 
\beq \label{eq:La-h}
\Rea \La^j =P_g^H(h^j).\eeq
We then write 
\beq \label{eq:proof-lemma-La} 
\La^j-\tfrac12\Th^j= \sum_{k=1}^{3(\gamma-1)} \i c_k^j \Om^k\eeq
in terms of the bases $\{\Om^j\}$ and $\{\Th^j\}$ described in Proposition \ref{prop:RT-new} and Lemma~\ref{lemma:Om}.
By definition $\La^j$ and $\half\Th^j$ induce the same change of the length coordinates, namely 
 $d\ell_i(\La^j)=\de_{i}^j=d\ell_i(\tfrac12\Th^j)$,
so  \eqref{eq:length-evol} implies that 
 $\Rea(b_0^i(\La^j-\tfrac12\Th^j))=0$ for every $i=1,\ldots, 3(\gamma-1)$.
 Since $b_0^i(\Om^k)=0$ if $i\neq k$ while $b_0^k(\Om^k)>0$, we thus find that the coefficients $c_k^j$ in the above expression must all be real, as claimed in the theorem. 
 
As a first step towards the proof of the estimate \eqref{est:coeff-lemma-La} for the $c_k^j$ claimed in the theorem, we show that 
$$\bar c\define  \max_{k} \abs{c_k^j}\leq C\ell_j$$
 by comparing the $L^2$-norms of the two sides of 
\eqref{eq:proof-lemma-La}. 
Remark \ref{rmk:lower_lin_comb} implies that 
\beqa \label{est:proof-c-initial}
\norm{\La^j-\thalf \Th^j}_{L^2(M,g)}^2=
\bnorm{\sum_k \i c_k^j \Om^k}_{L^2(M,g)}^2
\geq &\, \eps_1^2\cdot \sum_k \abs{c_k^j}^2 \geq \eps_1^2 \cdot \ov{c}^2
\eeqa
where $\eps_1>0$ is the lower bound on $\abs{b_0^k(\Om^k)} \cdot \norm{dz^2}_{L^2(\Col(\si^k))}$ obtained in \eqref{est:RT-neg3-newest}.
Using \eqref{eq:La-h} as well as \eqref{eq:proof-lemma-La}
we may furthermore write 
\beqa \label{eq:proof-La-2}
\norm{\La^j-\frac12\Th^j}_{L^2(M,g)}^2&=2
\norm{\Rea(\La^j)-\tfrac12\Rea(\Th^j)}_{L^2(M,g)}^2\\
&=
-2\langle \Rea(\La^j-\tfrac12 \Th^j), \tfrac12\Rea(\Th^j)\rangle +2\langle \Rea(\La^j-\tfrac12\Th^j), P^H_g(h^j)\rangle\\
&=- \sum_k c_k^j \langle \Rea(\i\Om^k), \Rea(\Th^j)\rangle +
2\sum_k c_k^j \langle \Rea(\i\Om^k),h^j\rangle\\
&\leq C\cdot \bigg( \sum_k \abs{\langle \Rea(\i\Om^k), \Rea(\Th^j)\rangle}+\sum_k \abs{\langle \Rea(\i\Om^k), h^j\rangle}\bigg)\cdot \bar c.
\eeqa
Once we bound the above sums, we thus obtain a bound that is linear in $\bar c$ which, combined with the bound on $\bar c^2$ that results from \eqref{est:proof-c-initial}, will allow us to bound $\bar c$.

To deal with the first sum, we 
note that the term with $k=j$ vanishes since 
 $\Th^j$ is a \textit{real} multiple of $\Om^j$ and since $\langle \Rea(\Om^j),\Rea(\i\Om^j)\rangle=0$, compare \eqref{eq:Re_inner_prod}.
 To estimate the other terms in the first sum 
we use the estimate \eqref{est:RT-neg4-new} on  $\Om^j=-\Th^j\norm{\Th^j}_{L^2(M,g)}^{-1}$ from Proposition~\ref{prop:RT-new} and the bound
\eqref{est:Th-L2-upper} on $\norm{\Th^j}_{L^2(M,g)}$ from Lemma \ref{lemma:Om}  to obtain that for every $k\neq j$ 
\beqas 
\abs{\langle \Rea(\i\Om^k), \Rea(\Th^j)\rangle}=
\norm{\Th^j}_{L^2(M,g)}\cdot \abs{\langle \Rea(\i\Om^k), \Rea(\Om^j)\rangle}\leq C \ell_j^{-1/2} \ell_j^{3/2}\ell_k^{3/2} \leq C\ell_j.
\eeqas
To estimate the terms in the second sum in \eqref{eq:proof-La-2} we 
first consider the subset 
$\Col_{c_5}(\si^j)$ of the collar $\Col(\si^j)$ introduced in Corollary~\ref{cor:h} on which $h^j$ is of the form \eqref{eq:form-h-collar}. 
We recall that horizontal tensors are trace-free and observe that 
for every circle  $\{s\}\times S^1$ contained in this set and every $k$
we obtain 
\beqas
\int_{\{s\}\times S^1}
\langle \Rea(\i \Om^k), h^j\rangle d\th& =
\sum_{n} \int_{S^1}\Rea \big({b_n^j(\i\Om^k)} e^{n(s+\i\th)}\big)\cdot  \xi_1(s) 2\rho^{-4}(s) d\th \\
&=-2\xi_1(s)\rho^{-4}(s)\text{Im}(b_0^j(\Om^k))2\pi=0
\eeqas  
as the principal parts of elements $\Om^k$ are real.
In particular $\langle  \Rea(\i \Om^k), h^j\rangle_{L^2(\Col_{c_5}(\si^j))}=0$ .

As Corollary \ref{cor:h} ensures that $\supp(h^j)\setminus \Col_{c_5}(\si^j)\subset \tilde \eta\thick (M,g)$, 
we may thus combine the bounds \eqref{est:h_L1}
for $h^j$ and \eqref{est:Linfty_Om_thick} for $\Om^k$ that are valid on this set to estimate 
\beqa \label{fast_fertig}
\abs{\langle  \Rea(\i \Om^k), h^j\rangle_{L^2(M,g)}}&=
\abs{\langle  \Rea(\i \Om^k), h^j\rangle_{L^2(\supp(h^j)\setminus \Col_{c_5}(\si^j) ,g)}}\\
&
\leq  \norm{h^j}_{L^1(M\setminus \Col_{c_5}(\si^j))} \cdot  \norm{\Rea \Om^k}_{L^\infty(\tilde \eta\thick(M,g))}\\  
&\leq C\ell_j \ell_k^{3/2}.
\eeqa
Hence also the second sum in 
 \eqref{eq:proof-La-2} is bounded by 
  $C\ell_j$ so we obtain 
  that $\norm{\La^j-\half \Th^j}_{L^2}^2\leq C\ell_j\bar c$. Combined with \eqref{est:proof-c-initial} this gives  
an initial bound of $\max_k\, |c_k^j|=\ov{c}\leq C\ell_j.$

For indices $k$ for which $\ell_k$ is bounded away from zero by some fixed constant this already yields the bound on $c_k^j$ claimed in the theorem.
For any other index, we multiply
\eqref{eq:proof-lemma-La} with $\i\Om^k$ to get
 \beqas 
\abs{c_k^j}\leq &   \half \abs{\Rea \langle \Th^j, \i\Om^k\rangle}+2\abs{\langle \Rea(\La^j),\Rea(\i\Om^k)\rangle} +\bar c \sum_{i\neq k} \abs{\langle\Om^i,\Om^k\rangle}
\\
\leq &C\ell_j \ell_k^{3/2}+ 2 \abs{\langle h^j,\Rea(\i\Om^k)\rangle}+
 C\ell_j \sum_{i\neq k}\ell_k^{3/2}  \ell_i^{3/2} 
 \leq  C\ell_j \ell_k^{3/2},
\eeqas 
see \eqref{fast_fertig} for the last step.
Having thus established \eqref{est:coeff-lemma-La}, we finally  remark that \eqref{est:La-minus-Th} is an immediate consequence of this estimate and the bound \eqref{est:Linfty_Om_thick} on $\norm{\Om^k}_{L^\infty(M,g)}$.
\end{proof}

\appendix

\section{Appendix:}\label{appendix:collar}
We will need the following `Collar lemma' throughout the paper.

\begin{lemma}[Keen-Randol \cite{randol}] \label{lemma:collar}
Let $(M,g)$ be a closed oriented hyperbolic surface and let $\si$ be a simple closed geodesic of length $\ell$. Then there is a neighbourhood $\Col(\si)$ around $\si$, a so-called collar, which is isometric to 
$\big((-X(\ell),X(\ell))\times S^1, \rho^2(s)(ds^2+d\theta^2)\big)$
where 
\beq \label{eq:rho-X} 
\rho(s)=\rho_\ell(s)=\frac{\ell}{2\pi \cos(\frac{\ell s}{2\pi})} 
\qquad\text{ and }\qquad  
X(\ell)=\frac{2\pi}{\ell}\left(\frac\pi2-\arctan\left(\sinh\left(\frac{\ell}{2}\right)\right) \right).\eeq
\end{lemma}

On collars we will always use the complex variable $z=s+\i\th$. 

We will use in particular the following properties of hyperbolic collars, and refer to 
\cite{Buser} as well as the appendices of \cite{RT-neg,RTZ,RT-horizontal} and the references therein for more information:

The width of a collar, i.e.~the distance $w_{\ell}\define \int_{-X(\ell)}^{X(\ell)} \rho(s) ds$ between the two boundary curves, is related to the length $\ell$ of the central geodesic by 
\beq\label{eq:width}
\sinh \tfrac{w_{\ell}}{2} \sinh \tfrac{\ell}{2}=1.
\eeq

The injectivity radius of points on the boundary curves of a collar is at least $\arsinh(1)$ 
and as the injectivity radii and conformal factors $\rho$ are of comparable size at points with bounded (euclidean) distance, we hence have that
\beq 
\label{est:rho-ends-of-collar}
\pi \rho(s)\geq \inj_g(s,\th)\geq c_\La>0 \text{ for all } \abs{s}\in [X(\ell)-\La,X(\ell))
\eeq
with $c_\La>0$ depending only on $\La$, compare e.g. \cite[(A.7)-(A.9)]{RT-neg}.

In our analysis of holomorphic quadratic differentials we use repeatedly that on a collar $\Col(\si)$ around a geodesic of length $\ell\in (0,2\arsinh (1))$ we have
\beqa
\label{est:sizes_on_collars}
|dz^2|_g=2\rho^{-2};
\qquad
& \|dz^2\|_{L^1(\Col(\si))}=8\pi X(\ell);\\
\qquad
\|dz^2\|_{L^\infty(\Col(\si))}=\frac{8\pi^2}{\ell^2};
\qquad & \|dz^2\|_{L^2(\Col(\si))}^2=\frac{32\pi^5}{\ell^3}-\frac{16\pi^4}{3}+O(\ell^2),
\eeqa
where norms on $\Col(\si)$ are always computed with respect to 
$g=\rho^2(ds^2+d\th^2)$.

We also remark that for every $\bar L$ there exists a constant $c_1=c_1(\bar L)>0$ so that if $\ell<\bar L$ then
\beq \label{est:dz-lower} \|dz^2\|_{L^2(\Col(\si))}\geq c_1\eeq
while an upper bound of the form 
\beq
\label{est:dz-L2-upper} 
\|dz^2\|_{L^2(\Col(\si))}\leq C\ell^{-3/2}\eeq
 holds true for a universal constant $C$.

As the principal part is orthogonal to the collar decay part we may combine the above estimates with 
\beq 
\label{est:b0-trivial} 
\abs{b_0(\Upsilon,\Col(\si))}\cdot \norm{dz^2}_{L^2(\Col(\si))}\leq \norm{\Upsilon}_{L^2(\Col(\si))}\leq \norm{\Upsilon}_{L^2(M,g)}
\eeq 
to obtain a trivial upper bound for the coefficient of the principal part on collars 
of 
\beq 
\label{est:b0-trivial-small} 
\abs{b_0(\Upsilon,\Col(\si))}\leq C \ell^{3/2} \norm{\Upsilon}_{L^2(M,g)}
\eeq 
so in particular 
 \beq 
\label{est:b0-trivial-2} 
\abs{b_0(\Upsilon,\Col(\si))}\leq C(\bar L) \norm{\Upsilon}_{L^2(\Col(\si))}\leq C(\bar L) \norm{\Upsilon}_{L^2(M,g)}
\eeq
for collars around geodesics of bounded length $L_g(\si)\leq \bar L$.

{\sc Mathematisches Institut, Universit{\"a}t Freiburg, 79104 Freiburg, Germany }\\
{\sc Mathematical Institute, University of Oxford, Oxford, OX2 6GG, UK}

\end{document}